\documentclass[12pt]{amsart}
\usepackage{amsmath,amsthm,amsfonts,amssymb,mathrsfs}
\usepackage{color}

\usepackage{tikz}

\usepackage{amssymb,cite}
\usepackage[colorlinks,plainpages,citecolor=magenta, linkcolor=blue, backref]{hyperref}

\usepackage{hyperref}
\date{\today}

\usepackage{hyperref}
\input{xy}
 \xyoption{all}
 \xyoption{arc}

\newtheorem{theorem}{Theorem}[section]

\newtheorem{proposition}[theorem]{Proposition}
\newtheorem{corollary}[theorem]{Corollary}

\newtheorem{lemma}[theorem]{Lemma}
\theoremstyle{definition}

\newtheorem{example}[theorem]{Example}
\newtheorem{remark}[theorem]{Remark}

\begin{document}

\title[On the group of automorphisms of the semigroup $\boldsymbol{B}_{\mathbb{Z}}^{\mathscr{F}}$ ]{On the group of automorphisms of the semigroup $\boldsymbol{B}_{\mathbb{Z}}^{\mathscr{F}}$ with the family $\mathscr{F}$ of inductive nonempty subsets of $\omega$}
\author{Oleg Gutik and Inna Pozdniakova}
\address{Ivan Franko National University of Lviv, Universytetska 1, Lviv, 79000, Ukraine}
\email{oleg.gutik@lnu.edu.ua, pozdnyakova.inna@gmail.com}



\begin{abstract}
We study automorphisms of the semigroup $\boldsymbol{B}_{Z\mathbb{}}^{\mathscr{F}}$ with the family $\mathscr{F}$ of inductive nonempty subsets of $\omega$ and prove that the group $\mathbf{Aut}(\boldsymbol{B}_{\mathbb{Z}}^{\mathscr{F}})$ of automorphisms of the semigroup $\boldsymbol{B}_{Z\mathbb{}}^{\mathscr{F}}$ is isomorphic to the additive group of integers.
\end{abstract}

\maketitle


\section{Introduction, motivation and main definitions}

We shall follow the terminology of~\cite{Clifford-Preston-1961, Clifford-Preston-1967, Harzheim=2005, Lawson=1998, Magnus-Karrass-Solitar-1976}. By $\omega$ we denote the set of all non-negative integers and by $\mathbb{Z}$ the set of all integers.

Let $\mathscr{P}(\omega)$ be  the family of all subsets of $\omega$. For any $F\in\mathscr{P}(\omega)$ and $n,m\in\omega$ we put $n-m+F=\{n-m+k\colon k\in F\}$ if $F\neq\varnothing$ and $n-m+\varnothing=\varnothing$. A subfamily $\mathscr{F}\subseteq\mathscr{P}(\omega)$ is called \emph{${\omega}$-closed} if $F_1\cap(-n+F_2)\in\mathscr{F}$ for all $n\in\omega$ and $F_1,F_2\in\mathscr{F}$. For any $a\in\omega$ we denote $[a)=\{x\in\omega\colon x\geqslant a\}$.

A subset $A$ of $\omega$ is said to be \emph{inductive}, if $i\in A$ implies $i+1\in A$. Obvious, that $\varnothing$ is an inductive subset of $\omega$.

\begin{remark}[\!\cite{Gutik-Mykhalenych=2021}]\label{remark-1.1}
\begin{enumerate}
  \item\label{remark-1.1(1)} By Lemma~6 from \cite{Gutik-Mykhalenych=2020} nonempty subset $F\subseteq \omega$ is inductive in $\omega$ if and only $(-1+F)\cap F=F$.
  \item\label{remark-1.1(2)} Since the set $\omega$ with the usual order is well-ordered, for any nonempty inductive subset $F$ in $\omega$ there exists nonnegative integer $n_F\in\omega$ such that $[n_F)=F$.
  \item\label{remark-1.1(3)} Statement \eqref{remark-1.1(2)} implies that the intersection of an arbitrary finite family of nonempty inductive subsets in $\omega$ is a nonempty inductive subset of  $\omega$.
\end{enumerate}
\end{remark}

A semigroup $S$ is called {\it inverse} if for any
element $x\in S$ there exists a unique $x^{-1}\in S$ such that
$xx^{-1}x=x$ and $x^{-1}xx^{-1}=x^{-1}$. The element $x^{-1}$ is
called the {\it inverse of} $x\in S$. If $S$ is an inverse
semigroup, then the function $\operatorname{inv}\colon S\to S$
which assigns to every element $x$ of $S$ its inverse element
$x^{-1}$ is called the {\it inversion}.

A \emph{partially ordered set} (or shortly a \emph{poset}) $(X,\leqq)$ is the set $X$ with the reflexive, antisymmetric and transitive relation $\leqq$. In this case relation $\leqq$ is called a partial order on $X$. A partially ordered set $(X,\leqq)$ is \emph{linearly ordered} or is a \emph{chain} if $x\leqq y$ or $y\leqq x$ for any $x,y\in X$. A map $f$ from a poset $(X,\leqq)$ onto a poset $(Y,\eqslantless)$ is said to be an order isomorphism if $f$ is bijective and $x\leqq y$ if and only if $f(x)\eqslantless f(y)$.

If $S$ is a semigroup, then we shall denote the subset of all
idempotents in $S$ by $E(S)$. If $S$ is an inverse semigroup, then
$E(S)$ is closed under multiplication and we shall refer to $E(S)$ as a
\emph{band} (or the \emph{band of} $S$). Then the semigroup
operation on $S$ determines the following partial order $\preccurlyeq$
on $E(S)$: $e\preccurlyeq f$ if and only if $ef=fe=e$. This order is
called the {\em natural partial order} on $E(S)$. A \emph{semilattice} is a commutative semigroup of idempotents.

If $S$ is an inverse semigroup then the semigroup operation on $S$ determines the following partial order $\preccurlyeq$
on $S$: $s\preccurlyeq t$ if and only if there exists $e\in E(S)$ such that $s=te$. This order is
called the {\em natural partial order} on $S$ \cite{Wagner-1952}.

The bicyclic monoid ${\mathscr{C}}(p,q)$ is the semigroup with the identity $1$ generated by two elements $p$ and $q$ subjected only to the condition $pq=1$. The semigroup operation on ${\mathscr{C}}(p,q)$ is determined as
follows:
\begin{equation*}
    q^kp^l\cdot q^mp^n=q^{k+m-\min\{l,m\}}p^{l+n-\min\{l,m\}}.
\end{equation*}
It is well known that the bicyclic monoid ${\mathscr{C}}(p,q)$ is a bisimple (and hence simple) combinatorial $E$-unitary inverse semigroup and every non-trivial congruence on ${\mathscr{C}}(p,q)$ is a group congruence \cite{Clifford-Preston-1961}.

On the set $\boldsymbol{B}_{\omega}=\omega\times\omega$ we define the semigroup operation ``$\cdot$'' in the following way
\begin{equation}\label{eq-1.1}
  (i_1,j_1)\cdot(i_2,j_2)=
  \left\{
    \begin{array}{ll}
      (i_1-j_1+i_2,j_2), & \hbox{if~} j_1\leqslant i_2;\\
      (i_1,j_1-i_2+j_2), & \hbox{if~} j_1\geqslant i_2.
    \end{array}
  \right.
\end{equation}
It is well known that the bicyclic monoid $\mathscr{C}(p,q)$ to the semigroup $\boldsymbol{B}_{\omega}$ is isomorphic by the mapping $\mathfrak{h}\colon \mathscr{C}(p,q)\to \boldsymbol{B}_{\omega}$, $q^kp^l\mapsto (k,l)$ (see: \cite[Section~1.12]{Clifford-Preston-1961} or \cite[Exercise IV.1.11$(ii)$]{Petrich-1984}).


Next we shall describe the construction which is introduced in \cite{Gutik-Mykhalenych=2020}.

Let $\boldsymbol{B}_{\omega}$ be the bicyclic monoid and $\mathscr{F}$ be an ${\omega}$-closed subfamily of $\mathscr{P}(\omega)$. On the set $\boldsymbol{B}_{\omega}\times\mathscr{F}$ we define the semigroup operation ``$\cdot$'' in the following way
\begin{equation}\label{eq-1.2}
  (i_1,j_1,F_1)\cdot(i_2,j_2,F_2)=
  \left\{
    \begin{array}{ll}
      (i_1-j_1+i_2,j_2,(j_1-i_2+F_1)\cap F_2), & \hbox{if~} j_1\leqslant i_2;\\
      (i_1,j_1-i_2+j_2,F_1\cap (i_2-j_1+F_2)), & \hbox{if~} j_1\geqslant i_2.
    \end{array}
  \right.
\end{equation}
In \cite{Gutik-Mykhalenych=2020} is proved that if the family $\mathscr{F}\subseteq\mathscr{P}(\omega)$ is ${\omega}$-closed then $(\boldsymbol{B}_{\omega}\times\mathscr{F},\cdot)$ is a semigroup. Moreover, if an ${\omega}$-closed family  $\mathscr{F}\subseteq\mathscr{P}(\omega)$ contains the empty set $\varnothing$ then the set
$ 
  \boldsymbol{I}=\{(i,j,\varnothing)\colon i,j\in\omega\}
$ 
is an ideal of the semigroup $(\boldsymbol{B}_{\omega}\times\mathscr{F},\cdot)$. For any ${\omega}$-closed family $\mathscr{F}\subseteq\mathscr{P}(\omega)$ the following semigroup
\begin{equation*}
  \boldsymbol{B}_{\omega}^{\mathscr{F}}=
\left\{
  \begin{array}{ll}
    (\boldsymbol{B}_{\omega}\times\mathscr{F},\cdot)/\boldsymbol{I}, & \hbox{if~} \varnothing\in\mathscr{F};\\
    (\boldsymbol{B}_{\omega}\times\mathscr{F},\cdot), & \hbox{if~} \varnothing\notin\mathscr{F}
  \end{array}
\right.
\end{equation*}
is defined in \cite{Gutik-Mykhalenych=2020}. The semigroup $\boldsymbol{B}_{\omega}^{\mathscr{F}}$ generalizes the bicyclic monoid and the countable semigroup of matrix units. It is proven in \cite{Gutik-Mykhalenych=2020} that $\boldsymbol{B}_{\omega}^{\mathscr{F}}$ is a combinatorial inverse semigroup and Green's relations, the natural partial order on $\boldsymbol{B}_{\omega}^{\mathscr{F}}$ and its set of idempotents are described.
Here, the criteria when the semigroup $\boldsymbol{B}_{\omega}^{\mathscr{F}}$ is simple, $0$-simple, bisimple, $0$-bisimple, or it has the identity, are given.
In particularly in \cite{Gutik-Mykhalenych=2020} it is proved that the semigroup $\boldsymbol{B}_{\omega}^{\mathscr{F}}$ is isomorphic to the semigrpoup of ${\omega}{\times}{\omega}$-matrix units if and only if $\mathscr{F}$ consists of a singleton set and the empty set, and $\boldsymbol{B}_{\omega}^{\mathscr{F}}$ is isomorphic to the bicyclic monoid if and only if $\mathscr{F}$ consists of a non-empty inductive subset of $\omega$.

Group congruences on the semigroup  $\boldsymbol{B}_{\omega}^{\mathscr{F}}$ and its homomorphic retracts  in the case when an ${\omega}$-closed family $\mathscr{F}$ consists of inductive non-empty subsets of $\omega$ are studied in \cite{Gutik-Mykhalenych=2021}. It is proven that a congruence $\mathfrak{C}$ on $\boldsymbol{B}_{\omega}^{\mathscr{F}}$ is a group congruence if and only if its restriction on a subsemigroup of $\boldsymbol{B}_{\omega}^{\mathscr{F}}$, which is isomorphic to the bicyclic semigroup, is not the identity relation. Also in \cite{Gutik-Mykhalenych=2021}, all non-trivial homomorphic retracts and isomorphisms  of the semigroup $\boldsymbol{B}_{\omega}^{\mathscr{F}}$ are described.

In \cite{Gutik-Lysetska=2021, Lysetska=2020} the algebraic structure of the semigroup $\boldsymbol{B}_{\omega}^{\mathscr{F}}$ is established in the case when ${\omega}$-closed family $\mathscr{F}$ consists of atomic subsets of ${\omega}$.

The set $\boldsymbol{B}_{\mathbb{Z}}=\mathbb{Z}\times\mathbb{Z}$ with the semigroup operation defined by formula \eqref{eq-1.1} is called the \emph{extended bicyclic semigroup} \cite{Warne-1968}. On the set $\boldsymbol{B}_{\mathbb{Z}}\times\mathscr{F}$, where $\mathscr{F}$ is an ${\omega}$-closed subfamily of $\mathscr{P}(\omega)$, we define the semigroup operation ``$\cdot$'' by formula \eqref{eq-1.2}. In \cite{Gutik-Pozdniakova=2021} it is proved that $(\boldsymbol{B}_{\mathbb{Z}}\times\mathscr{F},\cdot)$ is a semigroup. Moreover, if an ${\omega}$-closed family  $\mathscr{F}\subseteq\mathscr{P}(\omega)$ contains the empty set $\varnothing$ then the set
$ 
  \boldsymbol{I}=\{(i,j,\varnothing)\colon i,j\in\mathbb{Z}\}
$ 
is an ideal of the semigroup $(\boldsymbol{B}_{\mathbb{Z}}\times\mathscr{F},\cdot)$. For any ${\omega}$-closed family $\mathscr{F}\subseteq\mathscr{P}(\omega)$ the following semigroup
\begin{equation*}
  \boldsymbol{B}_{\mathbb{Z}}^{\mathscr{F}}=
\left\{
  \begin{array}{ll}
    (\boldsymbol{B}_{\mathbb{Z}}\times\mathscr{F},\cdot)/\boldsymbol{I}, & \hbox{if~} \varnothing\in\mathscr{F};\\
    (\boldsymbol{B}_{\mathbb{Z}}\times\mathscr{F},\cdot), & \hbox{if~} \varnothing\notin\mathscr{F}
  \end{array}
\right.
\end{equation*}
is defined in \cite{Gutik-Pozdniakova=2021} similarly as in \cite{Gutik-Mykhalenych=2020}. In \cite{Gutik-Pozdniakova=2021} it is proven that $\boldsymbol{B}_{\mathbb{Z}}^{\mathscr{F}}$ is a combinatorial inverse semigroup. Green's relations, the natural partial order on the semigroup $\boldsymbol{B}_{\mathbb{Z}}^{\mathscr{F}}$ and its set of idempotents are described.
Here, the criteria when the semigroup $\boldsymbol{B}_{\mathbb{Z}}^{\mathscr{F}}$ is simple, $0$-simple, bisimple, $0$-bisimple, is isomorphic to the extended bicyclic semigroup,  are derived. In particularly in \cite{Gutik-Pozdniakova=2021} it is proved that the semigroup $\boldsymbol{B}_{\mathbb{Z}}^{\mathscr{F}}$ is isomorphic to the semigrpoup of ${\omega}{\times}{\omega}$-matrix units if and only if $\mathscr{F}$ consists of a singleton set and the empty set, and $\boldsymbol{B}_{Z}^{\mathscr{F}}$ is isomorphic to the extended bicyclic semigroup if and only if $\mathscr{F}$ consists of a non-empty inductive subset of $\omega$.
Also, in \cite{Gutik-Pozdniakova=2021} it is proved that in the case when the family $\mathscr{F}$ consists of all singletons of $\omega$ and the empty set, the semigroup $\boldsymbol{B}_{\mathbb{Z}}^{\mathscr{F}}$ is isomorphic to the Brandt $\lambda$-extension of the semilattice $(\omega,\min)$, where $(\omega,\min)$ is the set $\omega$ with the semilattice operation $x\cdot y=\min\{x,y\}$.

It is well-known that every automorphism of the bicyclic monoid $\boldsymbol{B}_{\omega}$  is the identity self-map of $\boldsymbol{B}_{\omega}$ \cite{Clifford-Preston-1961}, and hence the group $\mathbf{Aut}(\boldsymbol{B}_{\omega})$ of automorphisms of $\boldsymbol{B}_{\omega}$ is trivial. The group $\mathbf{Aut}(\boldsymbol{B}_{\mathbb{Z}})$ of automorphisms of the extended bicyclic semigroup $\boldsymbol{B}_{\mathbb{Z}}$ is established in \cite{Gutik-Maksymyk-2017} and there it is proved that $\mathbf{Aut}(\boldsymbol{B}_{\mathbb{Z}})$ is isomorphic to the additive group of integers $\mathbb{Z}(+)$. Also in \cite{Gutik-Prokhorenkova-Sekh=2021} the semigroups of endomorphisms of the bicyclic  semigroup and  the extended bicyclic semigroup are described.

Later we assume that an ${\omega}$-closed family $\mathscr{F}$ consists of inductive nonempty subsets of $\omega$.

In this paper we study automorphisms of the semigroup $\boldsymbol{B}_{Z\mathbb{}}^{\mathscr{F}}$ with the family $\mathscr{F}$ of inductive nonempty subsets of $\omega$ and prove that the group $\mathbf{Aut}(\boldsymbol{B}_{\mathbb{Z}}^{\mathscr{F}})$ of automorphisms of the semigroup $\boldsymbol{B}_{Z\mathbb{}}^{\mathscr{F}}$ is isomorphic to the additive group of integers.

\section{Algebraic properties of the semigroup $\boldsymbol{B}_{\mathbb{Z}}^{\mathscr{F}}$}\label{section-2}

\begin{proposition}\label{proposition-2.1}
Let $\mathscr{F}$ be an arbitrary nonempty $\omega$-closed family of subsets of $\omega$ and let $n_0=\min\left\{\bigcup\mathscr{F}\right\}$. Then the following statements hold:
\begin{enumerate}
  \item\label{proposition-2.1(1)} $\mathscr{F}_0=\left\{-n_0+F\colon F\in\mathscr{F}\right\}$ is an $\omega$-closed family of subsets of $\omega$;
  \item\label{proposition-2.1(2)} the semigroups $\boldsymbol{B}_{\mathbb{Z}}^{\mathscr{F}}$ and $\boldsymbol{B}_{\mathbb{Z}}^{\mathscr{F}_0}$ are isomorphic by the mapping
\begin{equation*}
  (i,j,F)\mapsto(i,j,-n_0+F), \qquad i,j\in\mathbb{Z};
\end{equation*}
\end{enumerate}
\end{proposition}

\begin{proof}
Statement \eqref{proposition-2.1(1)} is proved in \cite[Proposition~1$(1)$]{Gutik-Mykhalenych=2021}. The proof of \eqref{proposition-2.1(2)} is similar to the one of Proposition~1$(2)$  from \cite{Gutik-Mykhalenych=2021}.
\end{proof}

Suppose that $\mathscr{F}$ is an $\omega$-closed family of inductive subsets of $\omega$. Fix an arbitrary $k\in\mathbb{Z}$. If $[0)\in\mathscr{F}$ and $[p)\in\mathscr{F}$ for some $p\in\omega$ then for any $i,j\in\mathbb{Z}$ and  we have that
\begin{align*}
  (k,k,[0))\cdot(i,j,[p)) &=
  \left\{
    \begin{array}{ll}
      (k-k+i,j,(k-i+[0))\cap[p)), & \hbox{if~} k<i;\\
      (k,j,[0)\cap[p))          , & \hbox{if~} k=i;\\
      (k,k-i+j,[0)\cap(i-k+[p))), & \hbox{if~} k>i
    \end{array}
  \right.=\\
  &=\left\{
    \begin{array}{ll}
      (i,j,[p)),                & \hbox{if~} k<i;\\
      (k,j,[p)),                & \hbox{if~} k=i;\\
      (k,k-i+j,[0)\cap[i-k+p)), & \hbox{if~} k>i
    \end{array}
  \right.
\end{align*}
and
\begin{align*}
(i,j,[p))\cdot(k,k,[0)) &=
  \left\{
    \begin{array}{ll}
      (i-j+k,k,(j-k+[p))\cap[0)), & \hbox{if~} j<k;\\
      (i,k,[p)\cap[0))          , & \hbox{if~} j=k;\\
      (i,j-k+k,[p)\cap(k-j+[0))), & \hbox{if~} j>k
    \end{array}
  \right.=\\
  &=\left\{
    \begin{array}{ll}
      (i-j+k,k,[j-k+p)\cap[0)), & \hbox{if~} j<k;\\
      (i,k,[p)),                & \hbox{if~} j=k;\\
      (i,j,[p),                 & \hbox{if~} j>k.
    \end{array}
  \right.
\end{align*}
Therefore the above equalities imply that
\begin{align*}
  (k,k,[0))\cdot \boldsymbol{B}_{\mathbb{Z}}^{\mathscr{F}}\cdot(k,k,[0)) &=(k,k,[0))\cdot \boldsymbol{B}_{\mathbb{Z}}^{\mathscr{F}}\cap \boldsymbol{B}_{\mathbb{Z}}^{\mathscr{F}}\cdot(k,k,[0))= \\
  &=\left\{(i,j,[p))\colon i,j\geqslant k, \; [p)\in\mathscr{F}\right\}
\end{align*}
for an arbitrary $k\in\mathbb{Z}$. We define $\boldsymbol{B}_{\mathbb{Z}}^{\mathscr{F}}[k,k,0)=(k,k,[0))\cdot \boldsymbol{B}_{\mathbb{Z}}^{\mathscr{F}}\cap \boldsymbol{B}_{\mathbb{Z}}^{\mathscr{F}}\cdot(k,k,[0))$. It is obvious that $\boldsymbol{B}_{\mathbb{Z}}^{\mathscr{F}}[k,k,0)$ is a subsemigroup  of $\boldsymbol{B}_{\mathbb{Z}}^{\mathscr{F}}$.

\begin{proposition}\label{proposition-2.2}
Let $\mathscr{F}$ be an arbitrary nonempty $\omega$-closed family of inductive nonempty subsets of $\omega$  such that $[0)\in\mathscr{F}$. Then the subsemigroup  $\boldsymbol{B}_{\mathbb{Z}}^{\mathscr{F}}[k,k,0)$ of $\boldsymbol{B}_{\mathbb{Z}}^{\mathscr{F}}$ is isomorphic to $\boldsymbol{B}_{\omega}^{\mathscr{F}}$.
\end{proposition}

\begin{proof}
Since the family $\mathscr{F}$ does not contain  the empty set, $\boldsymbol{B}_{\mathbb{Z}}^{\mathscr{F}}= (\boldsymbol{B}_{\mathbb{Z}}\times\mathscr{F},\cdot)$. We define a map $\mathfrak{I}\colon \boldsymbol{B}_{\omega}^{\mathscr{F}}\to \boldsymbol{B}_{\mathbb{Z}}^{\mathscr{F}}[k,k,0)$ in the following way $(i,j,[p))\mapsto(i+k,j+k,[p))$. It is obvious that $\mathfrak{I}$ is a bijection.  Then for any $i_1,i_2,j_1,j_2\in \mathbb{Z}$ and $F_1,F_2\in\mathscr{F}$ we have that
\begin{align*}
   \mathfrak{I}((i_1,j_1,F_1)\cdot(i_2,j_2,F_2))&=
  \left\{
    \begin{array}{ll}
       \mathfrak{I}(i_1-j_1+i_2,j_2,(j_1-i_2+F_1)\cap F_2), & \hbox{if~} j_1<i_2;\\
       \mathfrak{I}(i_1,j_2,F_1\cap F_2),                   & \hbox{if~} j_1=i_2;\\
       \mathfrak{I}(i_1,j_1-i_2+j_2,F_1\cap (i_2-j_1+F_2)), & \hbox{if~} j_1>i_2
    \end{array}
  \right.=\\
  &=
  \left\{
    \begin{array}{ll}
      (i_1-j_1+i_2+k,j_2+k,((j_1-i_2+F_1)\cap F_2)), & \hbox{if~} j_1<i_2;\\
      (i_1+k,j_2+k,F_1\cap F_2),                     & \hbox{if~} j_1=i_2;\\
      (i_1+k,j_1-i_2+j_2+k,(F_1\cap (i_2-j_1+F_2))), & \hbox{if~} j_1>i_2
    \end{array}
  \right.
\end{align*}
and
\begin{align*}
   \mathfrak{I}(i_1,j_1,F_1)&\cdot\mathfrak{I}(i_2+k,j_2+k,F_2)=(i_1+k,j_1+k,F_1)\cdot(i_2+k,j_2+k,F_2)=\\
  &=\left\{
    \begin{array}{ll}
       (i_1-j_1+i_2+k,j_2+k,(j_1-i_2+F_1)\cap F_2), & \hbox{if~} j_1+k<i_2+k;\\
       (i_1+k,j_2+k,F_1\cap F_2),                   & \hbox{if~} j_1+k=i_2+k;\\
       (i_1+k,j_1-i_2+j_2+k,F_1\cap(i_2-j_1+F_2)),  & \hbox{if~} j_1+k>i_2+k
    \end{array}
  \right. =  \\
  &=\left\{
    \begin{array}{ll}
       (i_1-j_1+i_2+k,j_2+k,(j_1-i_2+F_1)\cap F_2), & \hbox{if~} j_1<i_2;\\
       (i_1+k,j_2+k,F_1\cap F_2),                   & \hbox{if~} j_1=i_2;\\
       (i_1+k,j_1-i_2+j_2+k,F_1\cap(i_2-j_1+F_2)),  & \hbox{if~} j_1>i_2
    \end{array}
  \right.
 \end{align*}
and hence $\mathfrak{I}$ is a homomorphism which implies the statement of the proposition.
\end{proof}

By Remarks \ref{remark-1.1}\eqref{remark-1.1(2)} and \ref{remark-1.1}\eqref{remark-1.1(3)} every nonempty subset $F\in\mathscr{F}$ contains the least element, and hence later for every nonempty set $F\in\mathscr{F}$ we denote $n_F=\min F$.

\smallskip

Later we need the following lemma from \cite{Gutik-Mykhalenych=2021}.

\begin{lemma}[\!{\cite{Gutik-Mykhalenych=2021}}]\label{lemma-2.3}
Let $\mathscr{F}$ be an ${\omega}$-closed family of inductive subsets of $\omega$. Let $F_1$ and $F_2$ be elements of $\mathcal{F}$ such that $n_{F_1}<n_{F_2}$. Then for any positive integer $k\in\{n_{F_1}+1,\ldots,n_{F_2}-1\}$ there exists $F\in\mathscr{F}$ such that $F=[k)$.
\end{lemma}

Proposition~\ref{proposition-2.1} implies that without loss of generality later we may assume that $[0)\in \mathscr{F}$ for any ${\omega}$-closed family $\mathscr{F}$ of inductive subsets of $\omega$. Hence these arguments and Lemma~5 of \cite{Gutik-Mykhalenych=2020} imply the following proposition.

\begin{proposition}\label{proposition-2.4}
Let $\mathscr{F}$ be an infinite ${\omega}$-closed family of inductive nonempty subsets of $\omega$. Then the diagram on Fig.~\ref{fig-2.1}
\begin{figure}[h]
\vskip.1cm
\begin{center}
\tiny{
\begin{tikzpicture}[scale=.7]
\draw[>=latex,->,thick] (0,9.) -- (0,8.3);
\draw (0,8) node {$(-4,-4,[0))$};
\draw[>=latex,->,thick] (0,7.7) -- (0,6.3);
\draw[>=latex,->,thick] (.5,7.7) -- (1.5,7.3);
\draw (0,6) node {$(-3,-3,[0))$};
\draw[>=latex,->,thick] (0,5.7) -- (0,4.3);
\draw[>=latex,->,thick] (.5,5.7) -- (1.5,5.3);
\draw (0,4) node {$(-2,-2,[0))$};
\draw[>=latex,->,thick] (0,3.7) -- (0,2.3);
\draw[>=latex,->,thick] (.5,3.7) -- (1.5,3.3);
\draw (0,2) node {$(-1,-1,[0))$};
\draw[>=latex,->,thick] (0,1.7) -- (0,0.3);
\draw[>=latex,->,thick] (.5,1.7) -- (1.5,1.3);
\draw (0,0) node {$(0,0,[0))$};
\draw[>=latex,->,thick] (0,-.3) -- (0,-1.7);
\draw[>=latex,->,thick] (.5,-.3) -- (1.5,-.7);
\draw (0,-2) node {$(1,1,[0))$};
\draw[>=latex,->,thick] (0,-2.3) -- (0,-3.7);
\draw[>=latex,->,thick] (.5,-2.3) -- (1.5,-2.7);
\draw (0,-4) node {$(2,2,[0))$};
\draw[>=latex,->,thick] (0,-4.3) -- (0,-5.7);
\draw[>=latex,->,thick] (.5,-4.3) -- (1.5,-4.7);
\draw (0,-6) node {$(3,3,[0))$};
\draw[>=latex,->,thick] (0,-6.3) -- (0,-7.7);
\draw[>=latex,->,thick] (.5,-6.3) -- (1.5,-6.7);
\draw (0,-8) node {$(4,4,[0))$};
\draw[>=latex,->,thick] (0,-8.3) -- (0,-9.7);
\draw[>=latex,->,thick] (.5,-8.3) -- (1.5,-8.7);
\draw (0,-10) node {$\boldsymbol{\cdots}$};

\draw (2,9) node {$\boldsymbol{\cdots}$};
\draw[>=latex,->,thick] (2,8.7) -- (2,7.3);
\draw[>=latex,->,thick] (1.5,8.7) -- (.5,8.3);
\draw[>=latex,->,thick] (2.5,8.7) -- (3.5,8.3);
\draw (2,7) node {$(-4,-4,[1))$};
\draw[>=latex,->,thick] (2,6.7) -- (2,5.3);
\draw[>=latex,->,thick] (1.5,6.7) -- (.5,6.3);
\draw[>=latex,->,thick] (2.5,6.7) -- (3.5,6.3);
\draw (2,5) node {$(-3,-3,[1))$};
\draw[>=latex,->,thick] (2,4.7) -- (2,3.3);
\draw[>=latex,->,thick] (1.5,4.7) -- (.5,4.3);
\draw[>=latex,->,thick] (2.5,4.7) -- (3.5,4.3);
\draw (2,3) node {$(-2,-2,[1))$};
\draw[>=latex,->,thick] (2,2.7) -- (2,1.3);
\draw[>=latex,->,thick] (1.5,2.7) -- (.5,2.3);
\draw[>=latex,->,thick] (2.5,2.7) -- (3.5,2.3);
\draw (2,1) node {$(-1,-1,[1))$};
\draw[>=latex,->,thick] (2,.7) -- (2,-0.7);
\draw[>=latex,->,thick] (1.5,.7) -- (.5,.3);
\draw[>=latex,->,thick] (2.5,.7) -- (3.5,.3);
\draw (2,-1) node {$(0,0,[1))$};
\draw[>=latex,->,thick] (2,-1.3) -- (2,-2.7);
\draw[>=latex,->,thick] (1.5,-1.3) -- (.5,-1.7);
\draw[>=latex,->,thick] (2.5,-1.3) -- (3.5,-1.7);
\draw (2,-3) node {$(1,1,[1))$};
\draw[>=latex,->,thick] (2,-3.3) -- (2,-4.7);
\draw[>=latex,->,thick] (1.5,-3.3) -- (.5,-3.7);
\draw[>=latex,->,thick] (2.5,-3.3) -- (3.5,-3.7);
\draw (2,-5) node {$(2,2,[1))$};
\draw[>=latex,->,thick] (2,-5.3) -- (2,-6.7);
\draw[>=latex,->,thick] (1.5,-5.3) -- (.5,-5.7);
\draw[>=latex,->,thick] (2.5,-5.3) -- (3.5,-5.7);
\draw (2,-7) node {$(3,3,[1))$};
\draw[>=latex,->,thick] (2,-7.3) -- (2,-8.7);
\draw[>=latex,->,thick] (1.5,-7.3) -- (.5,-7.7);
\draw[>=latex,->,thick] (2.5,-7.3) -- (3.5,-7.7);
\draw (2,-9) node {$(4,4,[1))$};
\draw[>=latex,->,thick] (2,-9.3) -- (2,-10);
\draw[>=latex,->,thick] (1.5,-9.3) -- (.5,-9.7);
\draw[>=latex,->,thick] (2.5,-9.3) -- (3.5,-9.7);
\draw[>=latex,->,thick] (4,9.) -- (4,8.3);
\draw (4,8) node {$(-5,-5,[2))$};
\draw[>=latex,->,thick] (4,7.7) -- (4,6.3);
\draw[>=latex,->,thick] (4.5,7.7) -- (5.5,7.3);
\draw[>=latex,->,thick] (3.5,7.7) -- (2.5,7.3);
\draw (4,6) node {$(-4,-4,[2))$};
\draw[>=latex,->,thick] (4,5.7) -- (4,4.3);
\draw[>=latex,->,thick] (4.5,5.7) -- (5.5,5.3);
\draw[>=latex,->,thick] (3.5,5.7) -- (2.5,5.3);
\draw (4,4) node {$(-3,-3,[2))$};
\draw[>=latex,->,thick] (4,3.7) -- (4,2.3);
\draw[>=latex,->,thick] (4.5,3.7) -- (5.5,3.3);
\draw[>=latex,->,thick] (3.5,3.7) -- (2.5,3.3);
\draw (4,2) node {$(-2,-2,[2))$};
\draw[>=latex,->,thick] (4,1.7) -- (4,.3);
\draw[>=latex,->,thick] (4.5,1.7) -- (5.5,1.3);
\draw[>=latex,->,thick] (3.5,1.7) -- (2.5,1.3);
\draw (4,0) node {$(-1,-1,[2))$};
\draw[>=latex,->,thick] (4,-.3) -- (4,-1.7);
\draw[>=latex,->,thick] (4.5,-.3) -- (5.5,-.7);
\draw[>=latex,->,thick] (3.5,-.3) -- (2.5,-.7);
\draw (4,-2) node {$(0,0,[2))$};
\draw[>=latex,->,thick] (4,-2.3) -- (4,-3.7);
\draw[>=latex,->,thick] (4.5,-2.3) -- (5.5,-2.7);
\draw[>=latex,->,thick] (3.5,-2.3) -- (2.5,-2.7);
\draw (4,-4) node {$(1,1,[2))$};
\draw[>=latex,->,thick] (4,-4.3) -- (4,-5.7);
\draw[>=latex,->,thick] (4.5,-4.3) -- (5.5,-4.7);
\draw[>=latex,->,thick] (3.5,-4.3) -- (2.5,-4.7);
\draw (4,-6) node {$(2,2,[2))$};
\draw[>=latex,->,thick] (4,-6.3) -- (4,-7.7);
\draw[>=latex,->,thick] (4.5,-6.3) -- (5.5,-6.7);
\draw[>=latex,->,thick] (3.5,-6.3) -- (2.5,-6.7);
\draw (4,-8) node {$(3,3,[2))$};
\draw[>=latex,->,thick] (4,-8.3) -- (4,-9.7);
\draw[>=latex,->,thick] (4.5,-8.3) -- (5.5,-8.7);
\draw[>=latex,->,thick] (3.5,-8.3) -- (2.5,-8.7);
\draw (4,-10) node {$\boldsymbol{\cdots}$};
\draw (6,9) node {$\boldsymbol{\cdots}$};
\draw[>=latex,->,thick] (6,8.7) -- (6,7.3);
\draw[>=latex,->,thick] (5.5,8.7) -- (4.5,8.3);
\draw[>=latex,->,thick] (6.5,8.7) -- (7.5,8.3);
\draw (6,7) node {$(-5,-5,[3))$};
\draw[>=latex,->,thick] (6,6.7) -- (6,5.3);
\draw[>=latex,->,thick] (5.5,6.7) -- (4.5,6.3);
\draw[>=latex,->,thick] (6.5,6.7) -- (7.5,6.3);
\draw (6,5) node {$(-4,-4,[3))$};
\draw[>=latex,->,thick] (6,4.7) -- (6,3.3);
\draw[>=latex,->,thick] (5.5,4.7) -- (4.5,4.3);
\draw[>=latex,->,thick] (6.5,4.7) -- (7.5,4.3);
\draw (6,3) node {$(-3,-3,[3))$};
\draw[>=latex,->,thick] (6,2.7) -- (6,1.3);
\draw[>=latex,->,thick] (5.5,2.7) -- (4.5,2.3);
\draw[>=latex,->,thick] (6.5,2.7) -- (7.5,2.3);
\draw (6,1) node {$(-2,-2,[3))$};
\draw[>=latex,->,thick] (6,.7) -- (6,-.7);
\draw[>=latex,->,thick] (5.5,.7) -- (4.5,.3);
\draw[>=latex,->,thick] (6.5,.7) -- (7.5,.3);
\draw (6,-1) node {$(-1,-1,[3))$};
\draw[>=latex,->,thick] (6,-1.3) -- (6,-2.7);
\draw[>=latex,->,thick] (5.5,-1.3) -- (4.5,-1.7);
\draw[>=latex,->,thick] (6.5,-1.3) -- (7.5,-1.7);
\draw (6,-3) node {$(0,0,[3))$};
\draw[>=latex,->,thick] (6,-3.3) -- (6,-4.7);
\draw[>=latex,->,thick] (5.5,-3.3) -- (4.5,-3.7);
\draw[>=latex,->,thick] (6.5,-3.3) -- (7.5,-3.7);
\draw (6,-5) node {$(1,1,[3))$};
\draw[>=latex,->,thick] (6,-5.3) -- (6,-6.7);
\draw[>=latex,->,thick] (5.5,-5.3) -- (4.5,-5.7);
\draw[>=latex,->,thick] (6.5,-5.3) -- (7.5,-5.7);
\draw (6,-7) node {$(2,2,[3))$};
\draw[>=latex,->,thick] (6,-7.3) -- (6,-8.7);
\draw[>=latex,->,thick] (5.5,-7.3) -- (4.5,-7.7);
\draw[>=latex,->,thick] (6.5,-7.3) -- (7.5,-7.7);
\draw (6,-9) node {$(3,3,[3))$};
\draw[>=latex,->,thick] (6,-9.3) -- (6,-10);
\draw[>=latex,->,thick] (5.5,-9.3) -- (4.5,-9.7);
\draw[>=latex,->,thick] (6.5,-9.3) -- (7.5,-9.7);
\draw[>=latex,->,thick] (8,9.) -- (8,8.3);
\draw (8,8) node {$(-6,-6,[4))$};
\draw[>=latex,->,thick] (8,7.7) -- (8,6.3);
\draw[>=latex,->,thick] (8.5,7.7) -- (9.5,7.3);
\draw[>=latex,->,thick] (7.5,7.7) -- (6.5,7.3);
\draw (8,6) node {$(-5,-5,[4))$};
\draw[>=latex,->,thick] (8,5.7) -- (8,4.3);
\draw[>=latex,->,thick] (8.5,5.7) -- (9.5,5.3);
\draw[>=latex,->,thick] (7.5,5.7) -- (6.5,5.3);
\draw (8,4) node {$(-4,-4,[4))$};
\draw[>=latex,->,thick] (8,3.7) -- (8,2.3);
\draw[>=latex,->,thick] (8.5,3.7) -- (9.5,3.3);
\draw[>=latex,->,thick] (7.5,3.7) -- (6.5,3.3);
\draw (8,2) node {$(-3,-3,[4))$};
\draw[>=latex,->,thick] (8,1.7) -- (8,.3);
\draw[>=latex,->,thick] (8.5,1.7) -- (9.5,1.3);
\draw[>=latex,->,thick] (7.5,1.7) -- (6.5,1.3);
\draw (8,0) node {$(-2,-2,[4))$};
\draw[>=latex,->,thick] (8,-.3) -- (8,-1.7);
\draw[>=latex,->,thick] (8.5,-.3) -- (9.5,-.7);
\draw[>=latex,->,thick] (7.5,-.3) -- (6.5,-.7);
\draw (8,-2) node {$(-1,-1,[4))$};
\draw[>=latex,->,thick] (8,-2.3) -- (8,-3.7);
\draw[>=latex,->,thick] (8.5,-2.3) -- (9.5,-2.7);
\draw[>=latex,->,thick] (7.5,-2.3) -- (6.5,-2.7);
\draw (8,-4) node {$(0,0,[4))$};
\draw[>=latex,->,thick] (8,-4.3) -- (8,-5.7);
\draw[>=latex,->,thick] (8.5,-4.3) -- (9.5,-4.7);
\draw[>=latex,->,thick] (7.5,-4.3) -- (6.5,-4.7);
\draw (8,-6) node {$(1,1,[4))$};
\draw[>=latex,->,thick] (8,-6.3) -- (8,-7.7);
\draw[>=latex,->,thick] (8.5,-6.3) -- (9.5,-6.7);
\draw[>=latex,->,thick] (7.5,-6.3) -- (6.5,-6.7);
\draw (8,-8) node {$(2,2,[4))$};
\draw[>=latex,->,thick] (8,-8.3) -- (8,-9.7);
\draw[>=latex,->,thick] (8.5,-8.3) -- (9.5,-8.7);
\draw[>=latex,->,thick] (7.5,-8.3) -- (6.5,-8.7);
\draw (8,-10) node {$\boldsymbol{\cdots}$};
\draw (10,9) node {$\boldsymbol{\cdots}$};
\draw[>=latex,->,thick] (10,8.7) -- (10,7.3);
\draw[>=latex,->,thick] (9.5,8.7) -- (8.5,8.3);
\draw[>=latex,->,thick] (10.5,8.7) -- (11.5,8.3);
\draw (10,7) node {$(-6,-6,[5))$};
\draw[>=latex,->,thick] (10,6.7) -- (10,5.3);
\draw[>=latex,->,thick] (9.5,6.7) -- (8.5,6.3);
\draw[>=latex,->,thick] (10.5,6.7) -- (11.5,6.3);
\draw (10,5) node {$(-5,-5,[5))$};
\draw[>=latex,->,thick] (10,4.7) -- (10,3.3);
\draw[>=latex,->,thick] (9.5,4.7) -- (8.5,4.3);
\draw[>=latex,->,thick] (10.5,4.7) -- (11.5,4.3);
\draw (10,3) node {$(-4,-4,[5))$};
\draw[>=latex,->,thick] (10,2.7) -- (10,1.3);
\draw[>=latex,->,thick] (9.5,2.7) -- (8.5,2.3);
\draw[>=latex,->,thick] (10.5,2.7) -- (11.5,2.3);
\draw (10,1) node {$(-3,-3,[5))$};
\draw[>=latex,->,thick] (10,.7) -- (10,-.7);
\draw[>=latex,->,thick] (9.5,.7) -- (8.5,.3);
\draw[>=latex,->,thick] (10.5,.7) -- (11.5,.3);
\draw (10,-1) node {$(-2,-2,[5))$};
\draw[>=latex,->,thick] (10,-1.3) -- (10,-2.7);
\draw[>=latex,->,thick] (9.5,-1.3) -- (8.5,-1.7);
\draw[>=latex,->,thick] (10.5,-1.3) -- (11.5,-1.7);
\draw (10,-3) node {$(-1,-1,[5))$};
\draw[>=latex,->,thick] (10,-3.3) -- (10,-4.7);
\draw[>=latex,->,thick] (9.5,-3.3) -- (8.5,-3.7);
\draw[>=latex,->,thick] (10.5,-3.3) -- (11.5,-3.7);
\draw (10,-5) node {$(0,0,[5))$};
\draw[>=latex,->,thick] (10,-5.3) -- (10,-6.7);
\draw[>=latex,->,thick] (9.5,-5.3) -- (8.5,-5.7);
\draw[>=latex,->,thick] (10.5,-5.3) -- (11.5,-5.7);
\draw (10,-7) node {$(1,1,[5))$};
\draw[>=latex,->,thick] (10,-7.3) -- (10,-8.7);
\draw[>=latex,->,thick] (9.5,-7.3) -- (8.5,-7.7);
\draw[>=latex,->,thick] (10.5,-7.3) -- (11.5,-7.7);
\draw (10,-9) node {$(2,2,[5))$};
\draw[>=latex,->,thick] (10,-9.3) -- (10,-10);
\draw[>=latex,->,thick] (9.5,-9.3) -- (8.5,-9.7);
\draw[>=latex,->,thick] (10.5,-9.3) -- (11.5,-9.7);
\draw[>=latex,->,thick] (12,9.) -- (12,8.3);
\draw (12,8) node {$(-7,-7,[6))$};
\draw[>=latex,->,thick] (12,7.7) -- (12,6.3);
\draw[>=latex,->,thick] (12.5,7.7) -- (13.5,7.3);
\draw[>=latex,->,thick] (11.5,7.7) -- (10.5,7.3);
\draw (12,6) node {$(-6,-6,[6))$};
\draw[>=latex,->,thick] (12,5.7) -- (12,4.3);
\draw[>=latex,->,thick] (12.5,5.7) -- (13.5,5.3);
\draw[>=latex,->,thick] (11.5,5.7) -- (10.5,5.3);
\draw (12,4) node {$(-5,-5,[6))$};
\draw[>=latex,->,thick] (12,3.7) -- (12,2.3);
\draw[>=latex,->,thick] (12.5,3.7) -- (13.5,3.3);
\draw[>=latex,->,thick] (11.5,3.7) -- (10.5,3.3);
\draw (12,2) node {$(-4,-4,[6))$};
\draw[>=latex,->,thick] (12,1.7) -- (12,.3);
\draw[>=latex,->,thick] (12.5,1.7) -- (13.5,1.3);
\draw[>=latex,->,thick] (11.5,1.7) -- (10.5,1.3);
\draw (12,0) node {$(-3,-3,[6))$};
\draw[>=latex,->,thick] (12,-.3) -- (12,-1.7);
\draw[>=latex,->,thick] (12.5,-.3) -- (13.5,-.7);
\draw[>=latex,->,thick] (11.5,-.3) -- (10.5,-.7);
\draw (12,-2) node {$(-2,-2,[6))$};
\draw[>=latex,->,thick] (12,-2.3) -- (12,-3.7);
\draw[>=latex,->,thick] (12.5,-2.3) -- (13.5,-2.7);
\draw[>=latex,->,thick] (11.5,-2.3) -- (10.5,-2.7);
\draw (12,-4) node {$(-1,-1,[6))$};
\draw[>=latex,->,thick] (12,-4.3) -- (12,-5.7);
\draw[>=latex,->,thick] (12.5,-4.3) -- (13.5,-4.7);
\draw[>=latex,->,thick] (11.5,-4.3) -- (10.5,-4.7);
\draw (12,-6) node {$(0,0,[6))$};
\draw[>=latex,->,thick] (12,-6.3) -- (12,-7.7);
\draw[>=latex,->,thick] (12.5,-6.3) -- (13.5,-6.7);
\draw[>=latex,->,thick] (11.5,-6.3) -- (10.5,-6.7);
\draw (12,-8) node {$(1,1,[6))$};
\draw[>=latex,->,thick] (12,-8.3) -- (12,-9.7);
\draw[>=latex,->,thick] (12.5,-8.3) -- (13.5,-8.7);
\draw[>=latex,->,thick] (11.5,-8.3) -- (10.5,-8.7);
\draw (12,-10) node {$\boldsymbol{\cdots}$};
\draw (14,9) node {$\boldsymbol{\cdots}$};
\draw[>=latex,->,thick] (14,8.7) -- (14,7.3);
\draw[>=latex,->,thick] (13.5,8.7) -- (12.5,8.3);
\draw[>=latex,->,thick] (14.5,8.7) -- (15.5,8.3);
\draw (14,7) node {$(-7,-7,[7))$};
\draw[>=latex,->,thick] (14,6.7) -- (14,5.3);
\draw[>=latex,->,thick] (13.5,6.7) -- (12.5,6.3);
\draw[>=latex,->,thick] (14.5,6.7) -- (15.5,6.3);
\draw (14,5) node {$(-6,-6,[7))$};
\draw[>=latex,->,thick] (14,4.7) -- (14,3.3);
\draw[>=latex,->,thick] (13.5,4.7) -- (12.5,4.3);
\draw[>=latex,->,thick] (14.5,4.7) -- (15.5,4.3);
\draw (14,3) node {$(-5,-5,[7))$};
\draw[>=latex,->,thick] (14,2.7) -- (14,1.3);
\draw[>=latex,->,thick] (13.5,2.7) -- (12.5,2.3);
\draw[>=latex,->,thick] (14.5,2.7) -- (15.5,2.3);
\draw (14,1) node {$(-4,-4,[7))$};
\draw[>=latex,->,thick] (14,.7) -- (14,-.7);
\draw[>=latex,->,thick] (13.5,.7) -- (12.5,.3);
\draw[>=latex,->,thick] (14.5,.7) -- (15.5,.3);
\draw (14,-1) node {$(-3,-3,[7))$};
\draw[>=latex,->,thick] (14,-1.3) -- (14,-2.7);
\draw[>=latex,->,thick] (13.5,-1.3) -- (12.5,-1.7);
\draw[>=latex,->,thick] (14.5,-1.3) -- (15.5,-1.7);
\draw (14,-3) node {$(-2,-2,[7))$};
\draw[>=latex,->,thick] (14,-3.3) -- (14,-4.7);
\draw[>=latex,->,thick] (13.5,-3.3) -- (12.5,-3.7);
\draw[>=latex,->,thick] (14.5,-3.3) -- (15.5,-3.7);
\draw (14,-5) node {$(-1,-1,[7))$};
\draw[>=latex,->,thick] (14,-5.3) -- (14,-6.7);
\draw[>=latex,->,thick] (13.5,-5.3) -- (12.5,-5.7);
\draw[>=latex,->,thick] (14.5,-5.3) -- (15.5,-5.7);
\draw (14,-7) node {$(0,0,[7))$};
\draw[>=latex,->,thick] (14,-7.3) -- (14,-8.7);
\draw[>=latex,->,thick] (13.5,-7.3) -- (12.5,-7.7);
\draw[>=latex,->,thick] (14.5,-7.3) -- (15.5,-7.7);
\draw (14,-9) node {$(1,1,[7))$};
\draw[>=latex,->,thick] (14,-9.3) -- (14,-10);
\draw[>=latex,->,thick] (13.5,-9.3) -- (12.5,-9.7);
\draw[>=latex,->,thick] (14.5,-9.3) -- (15.5,-9.7);
\draw[>=latex,->,thick] (16,9.) -- (16,8.3);
\draw (16,8) node {$(-8,-8,[8))$};
\draw[>=latex,->,thick] (16,7.7) -- (16,6.3);
\draw[>=latex,->,thick] (16.5,7.7) -- (17.5,7.3);
\draw[>=latex,->,thick] (15.5,7.7) -- (14.5,7.3);
\draw (16,6) node {$(-7,-7,[8))$};
\draw[>=latex,->,thick] (16,5.7) -- (16,4.3);
\draw[>=latex,->,thick] (16.5,5.7) -- (17.5,5.3);
\draw[>=latex,->,thick] (15.5,5.7) -- (14.5,5.3);
\draw (16,4) node {$(-6,-6,[8))$};
\draw[>=latex,->,thick] (16,3.7) -- (16,2.3);
\draw[>=latex,->,thick] (16.5,3.7) -- (17.5,3.3);
\draw[>=latex,->,thick] (15.5,3.7) -- (14.5,3.3);
\draw (16,2) node {$(-5,-5,[8))$};
\draw[>=latex,->,thick] (16,1.7) -- (16,.3);
\draw[>=latex,->,thick] (16.5,1.7) -- (17.5,1.3);
\draw[>=latex,->,thick] (15.5,1.7) -- (14.5,1.3);
\draw (16,0) node {$(-4,-4,[8))$};
\draw[>=latex,->,thick] (16,-.3) -- (16,-1.7);
\draw[>=latex,->,thick] (16.5,-.3) -- (17.5,-.7);
\draw[>=latex,->,thick] (15.5,-.3) -- (14.5,-.7);
\draw (16,-2) node {$(-3,-3,[8))$};
\draw[>=latex,->,thick] (16,-2.3) -- (16,-3.7);
\draw[>=latex,->,thick] (16.5,-2.3) -- (17.5,-2.7);
\draw[>=latex,->,thick] (15.5,-2.3) -- (14.5,-2.7);
\draw (16,-4) node {$(-2,-2,[8))$};
\draw[>=latex,->,thick] (16,-4.3) -- (16,-5.7);
\draw[>=latex,->,thick] (16.5,-4.3) -- (17.5,-4.7);
\draw[>=latex,->,thick] (15.5,-4.3) -- (14.5,-4.7);
\draw (16,-6) node {$(-1,-1,[8))$};
\draw[>=latex,->,thick] (16,-6.3) -- (16,-7.7);
\draw[>=latex,->,thick] (16.5,-6.3) -- (17.5,-6.7);
\draw[>=latex,->,thick] (15.5,-6.3) -- (14.5,-6.7);
\draw (16,-8) node {$(0,0,[8))$};
\draw[>=latex,->,thick] (16,-8.3) -- (16,-9.7);
\draw[>=latex,->,thick] (16.5,-8.3) -- (17.5,-8.7);
\draw[>=latex,->,thick] (15.5,-8.3) -- (14.5,-8.7);
\draw (16,-10) node {$\boldsymbol{\cdots}$};
\draw (18,9) node {$\boldsymbol{\cdots}$};
\draw[>=latex,->,thick] (17.5,8.7) -- (16.5,8.3);
\draw (18,7) node {$\boldsymbol{\cdots}$};
\draw[>=latex,->,thick] (17.5,6.7) -- (16.5,6.3);
\draw (18,5) node {$\boldsymbol{\cdots}$};
\draw[>=latex,->,thick] (17.5,4.7) -- (16.5,4.3);
\draw (18,3) node {$\boldsymbol{\cdots}$};
\draw[>=latex,->,thick] (17.5,2.7) -- (16.5,2.3);
\draw (18,1) node {$\boldsymbol{\cdots}$};
\draw[>=latex,->,thick] (17.5,.7) -- (16.5,.3);
\draw (18,-1) node {$\boldsymbol{\cdots}$};
\draw[>=latex,->,thick] (17.5,-1.3) -- (16.5,-1.7);
\draw (18,-3) node {$\boldsymbol{\cdots}$};
\draw[>=latex,->,thick] (17.5,-3.3) -- (16.5,-3.7);
\draw (18,-5) node {$\boldsymbol{\cdots}$};
\draw[>=latex,->,thick] (17.5,-5.3) -- (16.5,-5.7);
\draw (18,-7) node {$\boldsymbol{\cdots}$};
\draw[>=latex,->,thick] (17.5,-7.3) -- (16.5,-7.7);
\draw (18,-9) node {$\boldsymbol{\cdots}$};
\draw[>=latex,->,thick] (17.5,-9.3) -- (16.5,-9.7);

\end{tikzpicture}
}
\end{center}
\vskip.1cm
\caption{The natural partial order on the band $E(\boldsymbol{B}_{\mathbb{Z}}^{\mathscr{F}})$}\label{fig-2.1}
\end{figure}
describes the natural partial order on the band of $\boldsymbol{B}_{\mathbb{Z}}^{\mathscr{F}}$.
\end{proposition}

By the similar way for a finite ${\omega}$-closed family of inductive nonempty subsets of $\omega$ we obtain the following

\begin{proposition}\label{proposition-2.4a}
Let $\mathscr{F}=\{[0),\ldots,[k)\}$. Then the diagram on Fig.~\ref{fig-2.1} without elements of the form $(i,j,[p))$ and their arrows, $i,j\in\mathbb{Z}$, $p>k$, describes the natural partial order on the band of $\boldsymbol{B}_{\mathbb{Z}}^{\mathscr{F}}$.
\end{proposition}

The definition of the semigroup operation in $\boldsymbol{B}_{\mathbb{Z}}^{\mathscr{F}}$ implies that in the case when $\mathscr{F}$ is an ${\omega}$-closed family subsets of $\omega$ and $F\in\mathscr{F}$ is a nonempty inductive subset in $\omega$ then the set
\begin{equation*}
  \boldsymbol{B}_{\mathbb{Z}}^{\{F\}}=\left\{(i,j,F)\colon i,j\in\mathbb{Z}\right\}
\end{equation*}
with the induced semigroup operation from $\boldsymbol{B}_{\mathbb{Z}}^{\mathscr{F}}$ is a subsemigroup of $\boldsymbol{B}_{\mathbb{Z}}^{\mathscr{F}}$ which by Proposition 5 from~\cite{Gutik-Pozdniakova=2021} is isomorphic to the extended bicyclic semigroup $\boldsymbol{B}_{\mathbb{Z}}$.

\begin{proposition}\label{proposition-2.5}
Let $\mathscr{F}$ be an arbitrary ${\omega}$-closed family of inductive subsets of $\omega$ and $S$ be a subsemigroup of $\boldsymbol{B}_{\mathbb{Z}}^{\mathscr{F}}$  which is isomorphic to the extended bicyclic semigroup $\boldsymbol{B}_{\mathbb{Z}}$. Then there exists a subset $F\in\mathscr{F}$ such that $S$ is a subsemigroup in $\boldsymbol{B}_{\mathbb{Z}}^{\{F\}}$.
\end{proposition}

\begin{proof}
Suppose that $\mathfrak{I}\colon \boldsymbol{B}_{\mathbb{Z}}\to S\subseteq \boldsymbol{B}_{\mathbb{Z}}^{\mathscr{F}}$ is an isomorphism. Proposition~21(2) of \cite[Section~1.4]{Lawson=1998} implies that the image $\mathfrak{I}(0,0)$ is an idempotent of $\boldsymbol{B}_{\mathbb{Z}}^{\mathscr{F}}$,
and hence by Lemma~1(2) from \cite{Gutik-Pozdniakova=2021}, $\mathfrak{I}(0,0)=(i,i,F)$ for some $i\in\mathbb{Z}$ and $F\in\mathscr{F}$.
By Proposition~2.1$(viii)$ of \cite{Fihel-Gutik=2011} the subset  $(0,0)\boldsymbol{B}_{\mathbb{Z}}(0,0)$ of $\boldsymbol{B}_{\mathbb{Z}}$ is isomorphic to the bicyclic semigroup, and hence the image $\mathfrak{I}\left((0,0)\boldsymbol{B}_{\mathbb{Z}}(0,0)\right)$ is isomorphic to the bicyclic semigroup $\boldsymbol{B}_{\omega}$. Then the definition of the natural partial order on $E(\boldsymbol{B}_{\mathbb{Z}}^{\mathscr{F}})$ and Corollary~1 from \cite{Gutik-Pozdniakova=2021} imply that there exists an integer $k$ such that $(i,i,F)\preccurlyeq(k,k,[0))$. By Proposition~\ref{proposition-2.2} the subsemigroup
\begin{equation*}
  \boldsymbol{B}_{\mathbb{Z}}^{\mathscr{F}}[k,k,0)=(k,k,[0))\cdot \boldsymbol{B}_{\mathbb{Z}}^{\mathscr{F}}\cdot(k,k,[0))
\end{equation*}
of $\boldsymbol{B}_{\mathbb{Z}}^{\mathscr{F}}$ is isomorphic to $\boldsymbol{B}_{\omega}^{\mathscr{F}}$. Since $(i,i,F)\preccurlyeq(k,k,[0))$ we have that  $\mathfrak{I}\left((0,0)\boldsymbol{B}_{\mathbb{Z}}(0,0)\right)\subseteq \boldsymbol{B}_{\mathbb{Z}}^{\mathscr{F}}[k,k,0)$, and hence $\mathfrak{I}\left((0,0)\boldsymbol{B}_{\mathbb{Z}}(0,0)\right)\subseteq \boldsymbol{B}_{\mathbb{Z}}^{\{F\}}$ by Proposition~4 of \cite{Gutik-Mykhalenych=2021}.

Next, fix any negative integer $n$. By Proposition~2.1$(viii)$ of \cite{Fihel-Gutik=2011} the subset  $(n,n)\boldsymbol{B}_{\mathbb{Z}}(n,n)$ of $\boldsymbol{B}_{\mathbb{Z}}$ is isomorphic to the bicyclic semigroup. Since $(0,0)\boldsymbol{B}_{\mathbb{Z}}(0,0)$ is an inverse subsemigroup of $(n,n)\boldsymbol{B}_{\mathbb{Z}}(n,n)$, the above arguments imply that $\mathfrak{I}\left((n,n)\boldsymbol{B}_{\mathbb{Z}}(n,n)\right)\subseteq \boldsymbol{B}_{\mathbb{Z}}^{\{F\}}$ for any negative integer $n$. Since
\begin{equation*}
  \boldsymbol{B}_{\mathbb{Z}}=\bigcup\left\{(k,k)\boldsymbol{B}_{\mathbb{Z}}(k,k)\colon -k\in\omega\right\},
\end{equation*}
we get that $\mathfrak{I}\left(\boldsymbol{B}_{\mathbb{Z}}\right)\subseteq \boldsymbol{B}_{\mathbb{Z}}^{\{F\}}$.
\end{proof}

\section{On authomorphisms of the semigroup $\boldsymbol{B}_{\mathbb{Z}}^{\mathscr{F}}$}\label{section-3}

Recall \cite{Green=1951} define relations $\mathscr{L}$ and $\mathscr{R}$ on an inverse semigroup $S$ by
\begin{equation*}
  (s,t)\in \mathscr{L} \quad \Leftrightarrow \quad s^{-1}s=t^{-1}t \qquad \hbox{and} \qquad (s,t)\in \mathscr{R} \quad \Leftrightarrow \quad ss^{-1}=tt^{-1}.
\end{equation*}
Both $\mathscr{L}$ and $\mathscr{R}$ are equivalence relations on $S$. The relation $\mathscr{D}$ is defined to be the smallest
equivalence relation which contains both $\mathscr{L}$ and $\mathscr{R}$, which is equivalent that $\mathscr{D}=\mathscr{L}\circ\mathscr{R}=\mathscr{R}\circ\mathscr{L}$ \cite{Lawson=1998}.

\begin{remark}\label{remark-3.1}
It is obvious that every semigroup  isomorphism $\mathfrak{i}\colon S\to T$ maps a $\mathscr{D}$-class (resp. $\mathscr{L}$-class, $\mathscr{R}$-class) of $S$ onto a $\mathscr{D}$-class (resp. $\mathscr{L}$-class, $\mathscr{R}$-class) of $T$.
\end{remark}

In this section we assume that $[0)\in \mathscr{F}$ for any ${\omega}$-closed family $\mathscr{F}$ of inductive subsets of $\omega$.

\smallskip

An automorphism $\mathfrak{a}$ of the semigroup $\boldsymbol{B}_{\mathbb{Z}}^{\mathscr{F}}$ is called a \emph{$(0,0,[0))$-automorphism} if $\mathfrak{a}(0,0,[0))=(0,0,[0))$.

\begin{theorem}\label{theorem-3.2}
Let $\mathscr{F}$ be an ${\omega}$-closed family of inductive nonempty subsets of $\omega$. Then every \emph{$(0,0,[0))$-automorphism} of the semigroup $\boldsymbol{B}_{\mathbb{Z}}^{\mathscr{F}}$ is the identity map.
\end{theorem}

\begin{proof}
Let $\mathfrak{a}\colon \boldsymbol{B}_{\mathbb{Z}}^{\mathscr{F}}\to \boldsymbol{B}_{\mathbb{Z}}^{\mathscr{F}}$ be an arbitrary $(0,0,[0))$-automorphism.

By Theorem~4$(iv)$ of \cite{Gutik-Pozdniakova=2021} the elements $(i_1,j_1,F_1)$ and $(i_2,j_2,F_2)$ of $\boldsymbol{B}_{\mathbb{Z}}^{\mathscr{F}}$ are $\mathscr{D}$-equivalent if and only if $F_1=F_2$. Since every automorphism preserves $\mathscr{D}$-classes, the above argument implies that $\mathfrak{a}(\boldsymbol{B}_{\mathbb{Z}}^{\{F_1\}})=\boldsymbol{B}_{\mathbb{Z}}^{\{F_2\}}$ if and only if $F_1=F_2$ for  $F_1,F_2\in\mathscr{F}$. Hence we have that $\mathfrak{a}(\boldsymbol{B}_{\mathbb{Z}}^{\{[0)\}})=\boldsymbol{B}_{\mathbb{Z}}^{\{[0)\}}$. By Proposition~21(6) of \cite[Section~1.4]{Lawson=1998} every  automorphism preserves the natural partial order on the semilattice $E(\boldsymbol{B}_{\mathbb{Z}}^{\mathscr{F}})$ and since $\mathfrak{a}$ is a $(0,0,[0))$-automorphism of $\boldsymbol{B}_{\mathbb{Z}}^{\mathscr{F}}$ we get that $\mathfrak{a}(i,i,[0))=(i,i,[0))$ for any integer $i$.

Fix arbitrary $k,l\in\mathbb{Z}$. Suppose that $\mathfrak{a}(k,l,[0))=(p,q,[0))$ for some integers $p$ and $q$. Since the semigroup $\boldsymbol{B}_{\mathbb{Z}}^{\mathscr{F}}$ is inverse, Proposition~21(1) of \cite[Section~1.4]{Lawson=1998} and Lemma~1(4) of \cite{Gutik-Pozdniakova=2021} imply that
\begin{equation*}
  \big(\mathfrak{a}(k,l,[0))\big)^{-1}=(p,q,[0))^{-1}=(q,p,[0)).
\end{equation*}
Again by Proposition~21(1) of \cite[Section~1.4]{Lawson=1998} we have that
\begin{align*}
  (k,k,[0)) &=\mathfrak{a}(k,k,[0))=\\
  &=\mathfrak{a}((k,l,[0))\cdot(l,k,[0)))=\\
  &=\mathfrak{a}(k,l,[0))\cdot\mathfrak{a}(l,k,[0))=\\
  &=\mathfrak{a}(k,l,[0))\cdot\mathfrak{a}\big((k,l,[0))^{-1}\big)=\\
  &=(p,q,[0))\cdot (q,p,[0))=\\
  &=(p,p,[0)),
\end{align*}
and hence $p=k$. By similar way we get that $l=q$. Therefore, $\mathfrak{a}(k,l,[0))=(k,l,[0))$ for any integers $k$ and $l$.

If $\mathscr{F}\neq\{[0)\}$ then by Lemma~\ref{lemma-2.3}, $[1)\in\mathscr{F}$. The definition of the natural partial order on the semilattice $E(\boldsymbol{B}_{\mathbb{Z}}^{\mathscr{F}})$ (also, see Proposition~\ref{proposition-2.4}) and Corollary~5 of \cite{Gutik-Pozdniakova=2021} imply that  $(0,0,[1))$ is the unique idempotent $\boldsymbol{\varepsilon}$ of the semigroup $\boldsymbol{B}_{\mathbb{Z}}^{\mathscr{F}}$ with the property
\begin{equation*}
  (1,1,[0))\preccurlyeq \boldsymbol{\varepsilon}\preccurlyeq(0,0,[0)).
\end{equation*}
Since by Proposition~21(6) of \cite[Section~1.4]{Lawson=1998} the automorphism $\mathfrak{a}$ preserves the natural partial order on the semilattice $E(\boldsymbol{B}_{\mathbb{Z}}^{\mathscr{F}})$, we get that  $\mathfrak{a}(0,0,[1))=(0,0,[1))$. Similar arguments as in the above paragraph imply that $\mathfrak{a}(k,l,[1))=(k,l,[1))$ for any integers $k$ and $l$.

Next, by induction we obtain that $\mathfrak{a}(k,l,[p))=(k,l,[p))$ for any $k,l\in\mathbb{Z}$ and $[p)\in\mathscr{F}$.
\end{proof}

\begin{proposition}\label{proposition-3.3}
Let $\mathscr{F}$ be an ${\omega}$-closed family of inductive nonempty subsets of $\omega$. Then  for every integer $k$ the map $\mathfrak{h}_k\colon \boldsymbol{B}_{\mathbb{Z}}^{\mathscr{F}}\to \boldsymbol{B}_{\mathbb{Z}}^{\mathscr{F}}$, $(i,j,[p))\mapsto (i+k,j+k,[p))$
is an automorphism of the semigroup $\boldsymbol{B}_{\mathbb{Z}}^{\mathscr{F}}$.
\end{proposition}

The proof of Proposition~\ref{proposition-3.3} is similar to Proposition~\ref{proposition-2.2}.

\smallskip

For a partially ordered set $(P, \leqq)$, a subset $X$ of $P$ is called \emph{order-convex}, if $x\leqq z\leqq y$ and ${x, y}\subset X$ implies that $z\in X$, for all $x, y, z\in P$ \cite{Harzheim=2005}.

\begin{lemma}\label{lemma-3.4}
If $\mathscr{F}$ is an infinite ${\omega}$-closed family of inductive nonempty subsets of $\omega$ then
\begin{equation*}
\left\{(0,0,[k))\colon k\in\omega\right\}
\end{equation*}
is an order-convex linearly ordered subset of $(E(\boldsymbol{B}_{\mathbb{Z}}^{\mathscr{F}}),\preccurlyeq)$.
\end{lemma}

\begin{proof}
Fix arbitrary $(0,0,[m)), (0,0,[n)), (0,0,[p))\in E(\boldsymbol{B}_{\mathbb{Z}}^{\mathscr{F}})$. If  $(0,0,[m))\preccurlyeq(0,0,[n))\preccurlyeq (0,0,[p))$ then Corollary~1 of \cite{Gutik-Pozdniakova=2021} implies that $[m)\subseteq [n)\subseteq [p)$. Hence we have that $m\geqslant n\geqslant p$, which implies the statement of the lemma.
\end{proof}

\begin{proposition}\label{proposition-3.5}
Let $\mathscr{F}$ be an infinite ${\omega}$-closed family of inductive nonempty subsets of $\omega$. Then
\begin{equation*}
  \mathfrak{a}(0,0,[0))\in \boldsymbol{B}_{\mathbb{Z}}^{\{[0)\}}
\end{equation*}
for any automorphism $\mathfrak{a}$ of the semigroup $\boldsymbol{B}_{\mathbb{Z}}^{\mathscr{F}}$.
\end{proposition}

\begin{proof}
Suppose to the contrary that there exists an automorphism $\mathfrak{a}$ of the semigroup $\boldsymbol{B}_{\mathbb{Z}}^{\mathscr{F}}$ such that $\mathfrak{a}(0,0,[0))\notin \boldsymbol{B}_{\mathbb{Z}}^{\{[0)\}}$. Then $\mathfrak{a}(0,0,[0))$ is an idempotent of the semigroup $\boldsymbol{B}_{\mathbb{Z}}^{\mathscr{F}}$.  Lemma~1(2) of \cite{Gutik-Pozdniakova=2021} implies that $\mathfrak{a}(0,0,[0))=(i,i,[p))$ for some integer $i$ and some positive integer $p$. Since the automorphism $\mathfrak{a}$ maps a $\mathscr{D}$-class of the semigroup $\boldsymbol{B}_{\mathbb{Z}}^{\mathscr{F}}$ onto its $\mathscr{D}$-class there exists an element $(0,0,[s))$ of the chain
\begin{equation}\label{eq-3.1}
  \cdots\preccurlyeq (0,0,[k))\preccurlyeq (0,0,[k-1))\preccurlyeq \cdots\preccurlyeq (0,0,[2))\preccurlyeq (0,0,[1))\preccurlyeq (0,0,[0))
\end{equation}
such that $\mathfrak{a}(0,0,[s))=(m,m,[0))\in \boldsymbol{B}_{\mathbb{Z}}^{\{[0)\}}$ for some integer $m$. By Proposition~21(6) of \cite[Section~1.4]{Lawson=1998} every  automorphism preserves the natural partial order on the semilattice $E(\boldsymbol{B}_{\mathbb{Z}}^{\mathscr{F}})$, and hence the inequality $(0,0,[s))\preccurlyeq(0,0,[0))$ implies that
\begin{equation*}
\mathfrak{a}(0,0,[s))=(m,m,[0))\preccurlyeq(i,i,[p))=\mathfrak{a}(0,0,[0)).
\end{equation*}
By Corollary~1 of \cite{Gutik-Pozdniakova=2021} we have that $m\geqslant i$ and $[0)\subseteq i-m+[p)$. The last inclusion implies that $m\geqslant i+p$. Since the chain \eqref{eq-3.1} is infinite and any its two distinct elements belong to distinct two $\mathscr{D}$-classes of the semigroup $\boldsymbol{B}_{\mathbb{Z}}^{\mathscr{F}}$, Proposition~21(6) of \cite[Section~1.4]{Lawson=1998} and Remark~\ref{remark-3.1} imply that there exists a positive integer $q>s$ such that $\mathfrak{a}(0,0,[q))=(t,t,[x))$ for some positive integer $x>p$ and some integer $t$. Then
\begin{equation*}
\mathfrak{a}(0,0,[q))=(t,t,[x))\preccurlyeq (m,m,[0))=\mathfrak{a}(0,0,[s))
\end{equation*}
and by Corollary~1 of \cite{Gutik-Pozdniakova=2021} we have that $t\geqslant m$ and $[x)\subseteq t-m+[0)$, and hence $x\geqslant t-m$.

Next we consider the idempotent $(i+1,i+1,[p))$ of the semigroup $\boldsymbol{B}_{\mathbb{Z}}^{\mathscr{F}}$. By Corollary~1 of \cite{Gutik-Pozdniakova=2021} we get that $(i+1,i+1,[p))\preccurlyeq(i,i,[p))$ in $E(\boldsymbol{B}_{\mathbb{Z}}^{\mathscr{F}})$. Since $x>p$ we have that $x\geqslant p+1$. The inequalities $t\geqslant m\geqslant i+p$ and $p\geqslant 1$ imply that $t\geqslant i+1$. Also, the inequalities $t\geqslant m\geqslant i$ and $x\geqslant p+1$ imply that $t+x\geqslant i+1+p$, and hence we obtain the inclusion $[x)\subseteq i+1-t+[p)$. By Corollary~1 of \cite{Gutik-Pozdniakova=2021} we have that $(t,t,[x))\preccurlyeq (i+1,i+1,[p))$. Since $\mathfrak{a}$ is an automorphism of the semigroup $\boldsymbol{B}_{\mathbb{Z}}^{\mathscr{F}}$, its restriction $\mathfrak{a}|_{E(\boldsymbol{B}_{\mathbb{Z}}^{\mathscr{F}})}\colon E(\boldsymbol{B}_{\mathbb{Z}}^{\mathscr{F}})\to E(\boldsymbol{B}_{\mathbb{Z}}^{\mathscr{F}})$ onto the band $E(\boldsymbol{B}_{\mathbb{Z}}^{\mathscr{F}})$ is an order automorphism of the partially ordered set $(E(\boldsymbol{B}_{\mathbb{Z}}^{\mathscr{F}}),\preccurlyeq)$, and hence the map $\mathfrak{a}|_{E(\boldsymbol{B}_{\mathbb{Z}}^{\mathscr{F}})}$ preserves order-convex subsets of $(E(\boldsymbol{B}_{\mathbb{Z}}^{\mathscr{F}}),\preccurlyeq)$. By Lemma~\ref{lemma-3.4}  chain \eqref{eq-3.1} is order-convex in the partially ordered set $(E(\boldsymbol{B}_{\mathbb{Z}}^{\mathscr{F}}),\preccurlyeq)$. The inequalities $(t,t,[x))\preccurlyeq (i+1,i+1,[p))\preccurlyeq(i,i,[p))$ in $E(\boldsymbol{B}_{\mathbb{Z}}^{\mathscr{F}})$ imply that the image of order-convex chain \eqref{eq-3.1} under the order automorphism $\mathfrak{a}|_{E(\boldsymbol{B}_{\mathbb{Z}}^{\mathscr{F}})}$  is not an order-convex subset of $(E(\boldsymbol{B}_{\mathbb{Z}}^{\mathscr{F}}),\preccurlyeq)$, a contradiction. The obtained contradiction implies the statement of the proposition.
\end{proof}

Later for any integer $k$ we assume that $\mathfrak{h}_k\colon \boldsymbol{B}_{\mathbb{Z}}^{\mathscr{F}}\to \boldsymbol{B}_{\mathbb{Z}}^{\mathscr{F}}$ is an automorphism of the semigroup $\boldsymbol{B}_{\mathbb{Z}}^{\mathscr{F}}$ defined in Proposition~\ref{proposition-3.3}.

\begin{theorem}\label{theorem-3.6}
Let $\mathscr{F}$ be an infinite ${\omega}$-closed family of inductive nonempty subsets of $\omega$. Then for any automorphism $\mathfrak{a}$ of the semigroup $\boldsymbol{B}_{\mathbb{Z}}^{\mathscr{F}}$ there exists an integer $p$ such that $\mathfrak{a}=\mathfrak{h}_p$.
\end{theorem}

\begin{proof}
By Proposition~\ref{proposition-3.5} there exists an integer $p$ such that $\mathfrak{a}(0,0,[0))=(-p,-p,[0))$. Then the composition $\mathfrak{h}_{p}\circ\mathfrak{a}$ is a $(0,0,[0))$-automorphism of the semigroup $\boldsymbol{B}_{\mathbb{Z}}^{\mathscr{F}}$, i.e., $(\mathfrak{h}_{p}\circ\mathfrak{a})(0,0,[0))=(0,0,[0))$, and hence by Theorem~\ref{theorem-3.2} the composition  $\mathfrak{h}_{p}\circ\mathfrak{a}$ is the identity map of $\boldsymbol{B}_{\mathbb{Z}}^{\mathscr{F}}$. Since $\mathfrak{h}_{p}$ and $\mathfrak{a}$ are bijections of $\boldsymbol{B}_{\mathbb{Z}}^{\mathscr{F}}$ the above arguments imply that $\mathfrak{a}=\mathfrak{h}_p$.
\end{proof}

Since $\mathfrak{h}_{k_1}\circ\mathfrak{h}_{k_2}=\mathfrak{h}_{k_1+k_2}$ and $\mathfrak{h}_{k_1}^{-1}=\mathfrak{h}_{-k_1}$, $k_1,k_2\in\mathbb{Z}$, for any automorphisms $\mathfrak{h}_{k_1}$ and $\mathfrak{h}_{k_2}$ of the semigroup $\boldsymbol{B}_{\mathbb{Z}}^{\mathscr{F}}$, Theorem~\ref{theorem-3.6} implies the following corollary.

\begin{corollary}\label{corollary-3.7}
Let $\mathscr{F}$ be an infinite ${\omega}$-closed family of inductive nonempty subsets of $\omega$. Then the group of automorphisms $\mathbf{Aut}(\boldsymbol{B}_{\mathbb{Z}}^{\mathscr{F}})$ of the semigroup $\boldsymbol{B}_{\mathbb{Z}}^{\mathscr{F}}$ is isomorphic to the additive group of integers $(\mathbb{Z},+)$.
\end{corollary}

The following example shows that for an arbitrary nonnegative integer $k$ and the finite family $\mathscr{F}=\{[0),[1),\ldots,[k)\}$ there exists an automorphism $\widetilde{\mathfrak{a}}\colon\boldsymbol{B}_{\mathbb{Z}}^{\mathscr{F}}\to \boldsymbol{B}_{\mathbb{Z}}^{\mathscr{F}}$ which is distinct from the form $\mathfrak{h}_p$.

\begin{example}\label{example-3.8}
Fix an arbitrary nonnegative integer $k$. Put
\begin{equation*}
\widetilde{\mathfrak{a}}(i,j,[s))=(i+s,j+s,[k-s))
\end{equation*}
for any $s=0,1,\ldots,k$ and all $i,j\in\mathbb{Z}$.
\end{example}

\begin{lemma}\label{lemma-3.9}
Let $k$ be an arbitrary nonnegative integer and $\mathscr{F}=\{[0),[1),\ldots,[k)\}$. Then $\widetilde{\mathfrak{a}}\colon\boldsymbol{B}_{\mathbb{Z}}^{\mathscr{F}}\to \boldsymbol{B}_{\mathbb{Z}}^{\mathscr{F}}$ is an automorphism.
\end{lemma}

\begin{proof}
Fix arbitrary $i,j,m,n\in\mathbb{Z}$. Without loss of generality we may assume that $s,t\in\{0,1,\ldots,k\}$ with $s<t$. Then we have that
\begin{align*}
\widetilde{\mathfrak{a}}((i,j,[s))\cdot(m,n,[t)))&=
\left\{
  \begin{array}{ll}
    \widetilde{\mathfrak{a}}(i-j+m,n,(j-m+[s))\cap[t)), & \hbox{if~} j<m;\\
    \widetilde{\mathfrak{a}}(i,n,[s)\cap[t)),           & \hbox{if~} j=m;\\
    \widetilde{\mathfrak{a}}(i,j-m+n,[s)\cap(m-j+[t))), & \hbox{if~} j>m
  \end{array}
\right.=
\\
&=
\left\{
  \begin{array}{ll}
    \widetilde{\mathfrak{a}}(i-j+m,n,[t)),     & \hbox{if~} j<m;\\
    \widetilde{\mathfrak{a}}(i,n,[1)),         & \hbox{if~} j=m;\\
    \widetilde{\mathfrak{a}}(i,j-m+n,[s)),     & \hbox{if~} j>m \hbox{~~and~~} m+t<j+s;\\
    \widetilde{\mathfrak{a}}(i,j-m+n,[s)),     & \hbox{if~} j>m \hbox{~~and~~} m+t=j+s;\\
    \widetilde{\mathfrak{a}}(i,j-m+n,m-j+[t)), & \hbox{if~} j>m \hbox{~~and~~} m+t>j+s
  \end{array}
\right.=
\\
&=
\left\{
  \begin{array}{ll}
   (i-j+m+s,n+s,[k-t)),     & \hbox{if~} j<m;\\
   (i+t,n+t,[k-t)),         & \hbox{if~} j=m;\\
   (i+s,j-m+n+s,[k-s)),     & \hbox{if~} j>m \hbox{~~and~~} m+t<j+s;\\
   (i+s,j-m+n+s,[k-s)),     & \hbox{if~} j>m \hbox{~~and~~} m+t=j+s;\\
   (i-j+m+t,n+t,[k-m+j-t)), & \hbox{if~} j>m \hbox{~~and~~} m+t>j+s,
  \end{array}
\right.
\end{align*}

\begin{align*}
\widetilde{\mathfrak{a}}(i,j,[s))\cdot&\widetilde{\mathfrak{a}}(m,n,[t))=
(i+s,j+s,[k-s))\cdot(m+t,n+t,[k-t))=
\\
=&
\left\{
  \begin{array}{ll}
    (i-j+m+t,n+t,(j+s-m-t+[k-s))\cap[k-t)), & \hbox{if~} j+s<m+t;\\
    (i+s,n+t,[k-s)\cap[k-t)),               & \hbox{if~} j+s=m+t;\\
    (i+s,j+s-m+n,[k-s)\cap(m+t-s-j+[k-t))), & \hbox{if~} j+s>m+t
  \end{array}
\right.=
\\
=&
\left\{
  \begin{array}{ll}
    (i-j+m+t,n+t,[k-t+j-m)\cap[k-t)), & \hbox{if~} j<m \hbox{~~and~~} j+s<m+t;\\
    (i+t,n+t,[k-t)\cap[k-t)),         & \hbox{if~} j=m \hbox{~~and~~} j+s<m+t;\\
    (i-j+m+t,n+t,[k-t+j-m)\cap[k-t)), & \hbox{if~} j>m \hbox{~~and~~} j+s<m+t;\\
    \hbox{vagueness},                 & \hbox{if~} j<m \hbox{~~and~~} j+s=m+t;\\
    \hbox{vagueness},                 & \hbox{if~} j=m \hbox{~~and~~} j+s=m+t;\\
    (i+s,n+t,[k-t)),                  & \hbox{if~} j>m \hbox{~~and~~} j+s=m+t;\\
    \hbox{vagueness},                 & \hbox{if~} j<m \hbox{~~and~~} j+s>m+t;\\
    \hbox{vagueness},                 & \hbox{if~} j=m \hbox{~~and~~} j+s>m+t;\\
    (i+s,j-m+n+s,[k-s)\cap[k-s-j+m)), & \hbox{if~} j>m \hbox{~~and~~} j+s>m+t
  \end{array}
\right.=
\end{align*}

\begin{align*}
=&
\left\{
  \begin{array}{ll}
    (i-j+m+t,n+t,[k-t)),       & \hbox{if~} j<m \hbox{~~and~~} j+s<m+t;\\
    \hbox{vagueness},          & \hbox{if~} j<m \hbox{~~and~~} j+s=m+t;\\
    \hbox{vagueness},          & \hbox{if~} j<m \hbox{~~and~~} j+s>m+t;\\
    (i+t,n+t,[k-t)),           & \hbox{if~} j=m \hbox{~~and~~} j+s<m+t;\\
    \hbox{vagueness},          & \hbox{if~} j=m \hbox{~~and~~} j+s=m+t;\\
    \hbox{vagueness},          & \hbox{if~} j=m \hbox{~~and~~} j+s>m+t;\\
    (i-j+m+t,n+t,[k-t+j-m)),    & \hbox{if~} j>m \hbox{~~and~~} j+s<m+t;\\
    (i+s,n+t,[k-t)),           & \hbox{if~} j>m \hbox{~~and~~} j+s=m+t;\\
    (i+s,j-m+n+s,[k-s)),       & \hbox{if~} j>m \hbox{~~and~~} j+s>m+t,
  \end{array}
\right.
\end{align*}

\begin{align*}
\widetilde{\mathfrak{a}}((m,n,[t))\cdot(i,j,[s))&=
\left\{
  \begin{array}{ll}
    \widetilde{\mathfrak{a}}(m-n+i,j,(n-i+[t))\cap[s)), & \hbox{if~} n<i;\\
    \widetilde{\mathfrak{a}}(m,j,[t)\cap[s)),           & \hbox{if~} n=i;\\
    \widetilde{\mathfrak{a}}(m,n-i+j,[t)\cap(i-n+[s))), & \hbox{if~} n>i
  \end{array}
\right.=
\\
&=
\left\{
  \begin{array}{ll}
    \widetilde{\mathfrak{a}}(m-n+i,j,[s)),       & \hbox{if~} n<i \hbox{~~and~~} n+t<i+s;\\
    \widetilde{\mathfrak{a}}(m-n+i,j,[s)),       & \hbox{if~} n<i \hbox{~~and~~} n+t=i+s;\\
    \widetilde{\mathfrak{a}}(m-n+i,j,[n-i+t)),   & \hbox{if~} n<i \hbox{~~and~~} n+t>i+s;\\
    \widetilde{\mathfrak{a}}(m,j,[t)),           & \hbox{if~} n=i;\\
    \widetilde{\mathfrak{a}}(m,n-i+j,[t)),        & \hbox{if~} n>i
  \end{array}
\right.=\\
&=
\left\{
  \begin{array}{ll}
    (m-n+i+s,j+s,[k-s)),              & \hbox{if~} n<i \hbox{~~and~~} n+t<i+s;\\
    (m-n+i+s,j+s,[k-s)),              & \hbox{if~} n<i \hbox{~~and~~} n+t=i+s;\\
    (m+t,j+[n-i+t,[k-t+i-n)),         & \hbox{if~} n<i \hbox{~~and~~} n+t>i+s;\\
    (m+t,j+t,[k-t)),                  & \hbox{if~} n=i;\\
    (m+t,n-i+j+t,[k-t)),               & \hbox{if~} n>i,
  \end{array}
\right.
\end{align*}

\begin{align*}
\widetilde{\mathfrak{a}}(m,n,[t))&\cdot\widetilde{\mathfrak{a}}(i,j,[s))=
(m+t,n+t,[k-t))\cdot(i+s,j+s,[k-s))=
\\
=&
\left\{
  \begin{array}{ll}
    (m-n+i+s,j+s,(n+t-i-s+[k-t))\cap[k-s)), & \hbox{if~} n+t<i+s;\\
    (m+t,j+s,[k-t)\cap[k-s)),               & \hbox{if~} n+t=i+s;\\
    (m+t,n-i+j+t,[k-t)\cap(i+s-n-t+[k-s))), & \hbox{if~} n+t>i+s
  \end{array}
\right.=
\\
=&
\left\{
  \begin{array}{ll}
    (m-n+i+s,j+s,[k-s+n-i))\cap[k-s)), & \hbox{if~} n+t<i+s;\\
    (m+t,j+s,[k-s)),                   & \hbox{if~} n+t=i+s;\\
    (m+t,n-i+j+t,[k-t+i-n)\cap[k-t)),  & \hbox{if~} n+t>i+s
  \end{array}
\right.=
\\
=&
\left\{
  \begin{array}{ll}
    (m-n+i+s,j+s,[k-s)),      & \hbox{if~} n<i \hbox{~~and~~} n+t<i+s;\\
    \hbox{vagueness},         & \hbox{if~} n=i \hbox{~~and~~} n+t<i+s;\\
    \hbox{vagueness},         & \hbox{if~} n>i \hbox{~~and~~} n+t<i+s;\\
    (m+t,j+s,[k-s)),          & \hbox{if~} n<i \hbox{~~and~~} n+t=i+s;\\
    \hbox{vagueness},         & \hbox{if~} n=i \hbox{~~and~~} n+t=i+s;\\
    \hbox{vagueness},         & \hbox{if~} n>i \hbox{~~and~~} n+t=i+s;\\
    (m+t,n-i+j+t,[k-t+i-n)),  & \hbox{if~} n<i \hbox{~~and~~} n+t>i+s;\\
    (m+t,n-i+j+t,[k-t)),      & \hbox{if~} n=i \hbox{~~and~~} n+t>i+s;\\
    (m+t,n-i+j+t,[k-t)),      & \hbox{if~} n>i \hbox{~~and~~} n+t>i+s
  \end{array}
\right.=
\\
\end{align*}

\begin{align*}
=&
\left\{
  \begin{array}{ll}
    (m-n+i+s,j+s,[k-s)),      & \hbox{if~} n<i \hbox{~~and~~} n+t<i+s;\\
    (m+t,j+s,[k-s)),          & \hbox{if~} n<i \hbox{~~and~~} n+t=i+s;\\
    (m+t,n-i+j+t,[k-t+i-n)),  & \hbox{if~} n<i \hbox{~~and~~} n+t>i+s;\\
    \hbox{vagueness},         & \hbox{if~} n=i \hbox{~~and~~} n+t<i+s;\\
    \hbox{vagueness},         & \hbox{if~} n=i \hbox{~~and~~} n+t=i+s;\\
    (m+t,n-i+j+t,[k-t)),      & \hbox{if~} n=i \hbox{~~and~~} n+t>i+s;\\
    \hbox{vagueness},         & \hbox{if~} n>i \hbox{~~and~~} n+t<i+s;\\
    \hbox{vagueness},         & \hbox{if~} n>i \hbox{~~and~~} n+t=i+s;\\
    (m+t,n-i+j+t,[k-t)),      & \hbox{if~} n>i \hbox{~~and~~} n+t>i+s,
  \end{array}
\right.
\end{align*}

\begin{align*}
\widetilde{\mathfrak{a}}((i,j,[s))\cdot(m,n,[s))&=
\left\{
  \begin{array}{ll}
    \widetilde{\mathfrak{a}}(i-j+m,n,(j-m+[s))\cap[s)), & \hbox{if~} j<m;\\
    \widetilde{\mathfrak{a}}(i,n,[s)\cap[s)),           & \hbox{if~} j=m;\\
    \widetilde{\mathfrak{a}}(i,j-m+n,[s)\cap(m-j+[s))), & \hbox{if~} j>m
  \end{array}
\right.=
\\
&=
\left\{
  \begin{array}{ll}
    \widetilde{\mathfrak{a}}(i-j+m,n,[s)), & \hbox{if~} j<m;\\
    \widetilde{\mathfrak{a}}(i,n,[s)),     & \hbox{if~} j=m;\\
    \widetilde{\mathfrak{a}}(i,j-m+n,[s)), & \hbox{if~} j>m
  \end{array}
\right.=
\\
&=
\left\{
  \begin{array}{ll}
   (i-j+m+s,n+s,[k-s)), & \hbox{if~} j<m;\\
   (i+s,n+s,[k-s)),     & \hbox{if~} j=m;\\
   (i+s,j-m+n+s,[k-s)), & \hbox{if~} j>m,
  \end{array}
\right.
\end{align*}

\begin{align*}
\widetilde{\mathfrak{a}}(i,j,[s))&\cdot\widetilde{\mathfrak{a}}(m,n,[s))=
(i+s,j+s,[k-s))\cdot(m+s,n+s,[k-s))=
\\
=&
\left\{
  \begin{array}{ll}
    (i-j+m+s,n+s,(j-m+[k-s))\cap[k-s)), & \hbox{if~} j+s<m+s;\\
    (i+s,n+s,[k-s)\cap[k-s)),           & \hbox{if~} j+s=m+s;\\
    (i+s,j-m+n+s,[k-s)\cap(m-j+[k-s))), & \hbox{if~} j+s>m+s
  \end{array}
\right.=
\\
=&
\left\{
  \begin{array}{ll}
    (i-j+m+s,n+s,[k-s)), & \hbox{if~} j<m;\\
    (i+s,n+s,[k-s)),     & \hbox{if~} j=m;\\
    (i+s,j-m+n+s,[k-s),  & \hbox{if~} j>m.
  \end{array}
\right.
\end{align*}

The above equalities imply that the map $\widetilde{\mathfrak{a}}\colon\boldsymbol{B}_{\mathbb{Z}}^{\mathscr{F}}\to \boldsymbol{B}_{\mathbb{Z}}^{\mathscr{F}}$ is an endomorphism, and since $\widetilde{\mathfrak{a}}$ is bijective, it is an automorphism of the semigroup $\boldsymbol{B}_{\mathbb{Z}}^{\mathscr{F}}$.
\end{proof}

\begin{proposition}\label{proposition-3.10}
Let $k$ be any positive integer and $\mathscr{F}=\{[0),\ldots,[k)\}$. Then either $\mathfrak{a}(0,0,[0))\in \boldsymbol{B}_{\mathbb{Z}}^{\{[0)\}}$ or $\mathfrak{a}(0,0,[0))\in \boldsymbol{B}_{\mathbb{Z}}^{\{[k)\}}$ for any automorphism $\mathfrak{a}$ of the semigroup $\boldsymbol{B}_{\mathbb{Z}}^{\mathscr{F}}$.
\end{proposition}

\begin{proof}
Suppose to the contrary that there exists a positive integer $m<k$ such that $\mathfrak{a}(0,0,[0))\in \boldsymbol{B}_{\mathbb{Z}}^{\{[m)\}}$.  Since $\mathfrak{a}(0,0,[0))$ is an idempotent of $\boldsymbol{B}_{\mathbb{Z}}^{\mathscr{F}}$, by Lemma~1(2) of \cite{Gutik-Pozdniakova=2021} there exists an integer $p$ such that $\mathfrak{a}(0,0,[0))=(p,p,[m))$. Then by the order convexity of the subset $L_1=\{(0,0,[0)), (0,0,[1))\}$ of $E(\boldsymbol{B}_{\mathbb{Z}}^{\mathscr{F}})$ we obtain  that the image $\mathfrak{a}(L_1)$ is an order convex chain in $E(\boldsymbol{B}_{\mathbb{Z}}^{\mathscr{F}})$ with the respect to the natural partial order. Then Remark~\ref{remark-3.1} and the description of the natural partial order on $E(\boldsymbol{B}_{\mathbb{Z}}^{\mathscr{F}})$ (see: Proposition~\ref{proposition-2.4a}) imply that either $\mathfrak{a}(0,0,[1))=(p,p,[m+1))$ or $\mathfrak{a}(0,0,[1))=(p+1,p+1,[m-1))$.

Suppose that the equality $\mathfrak{a}(0,0,[1))=(p,p,[m+1))$ holds. If $m+1=k$ then the equalities $0<m<k$ and Remark~\ref{remark-3.1} imply that $\mathfrak{a}(0,0,[2))\in \boldsymbol{B}_{\mathbb{Z}}^{\mathscr{F}}\setminus \boldsymbol{B}_{\mathbb{Z}}^{\{[k-1),[k)\}}$. Since $(0,0,[2))\preccurlyeq(0,0,[1))\preccurlyeq(0,0,[0))$, Proposition~21(6) of \cite[Section~1.4]{Lawson=1998} implies that $\mathfrak{a}(0,0,[2))\preccurlyeq\mathfrak{a}(0,0,[1))\preccurlyeq\mathfrak{a}(0,0,[0))$. Then  $\{\mathfrak{a}(0,0,[0)),\mathfrak{a}(0,0,[1)),\mathfrak{a}(0,0,[2))\}$ is not an order convex subset of $(E(\boldsymbol{B}_{\mathbb{Z}}^{\mathscr{F}}),\preccurlyeq)$, because $\mathfrak{a}(0,0,[2))\in \boldsymbol{B}_{\mathbb{Z}}^{\mathscr{F}}\setminus \boldsymbol{B}_{\mathbb{Z}}^{\{[k-1),[k)\}}$, a contradiction, and hence we obtain that $m+1<k$.

The above arguments and induction imply that there exists a positive integer $n_0<k$ such that $\mathfrak{a}(0,0,[n_0))=(p,p,[k))$. Then $\mathfrak{a}(0,0,[n_0+1))\in \boldsymbol{B}_{\mathbb{Z}}^{\mathscr{F}}\setminus \boldsymbol{B}_{\mathbb{Z}}^{\{[m),\ldots,[k)\}}$ and by the description of the natural partial order on $E(\boldsymbol{B}_{\mathbb{Z}}^{\mathscr{F}})$ (see: Proposition~\ref{proposition-2.4}) we get that
\begin{equation*}
\{\mathfrak{a}(0,0,[0)),\mathfrak{a}(0,0,[1)),\ldots,\mathfrak{a}(0,0,[n_0)),\mathfrak{a}(0,0,[n_0+1))\}
\end{equation*}
is not an order convex subset of $(E(\boldsymbol{B}_{\mathbb{Z}}^{\mathscr{F}}),\preccurlyeq)$, a contradiction. The obtained contradiction implies that $\mathfrak{a}(0,0,[1))\neq(p,p,[m+1))$.

In the case $\mathfrak{a}(0,0,[1))=(p+1,p+1,[m-1))$ by similar way we get a contradiction.
\end{proof}

Later we assume that $\mathfrak{h}_p$ and $\widetilde{\mathfrak{a}}$ are automorphisms of the semigroup $\boldsymbol{B}_{\mathbb{Z}}^{\mathscr{F}}$ defined in Proposition~\ref{proposition-3.3} and Example~\ref{example-3.8}, respectively.

\begin{proposition}\label{proposition-3.11}
Let $k$ be any positive integer and $\mathscr{F}=\{[0),\ldots,[k)\}$. Let $\mathfrak{a}\colon\boldsymbol{B}_{\mathbb{Z}}^{\mathscr{F}}\to \boldsymbol{B}_{\mathbb{Z}}^{\mathscr{F}}$ be an automorphisms such that $\mathfrak{a}(0,0,[0))\in \boldsymbol{B}_{\mathbb{Z}}^{\{[k)\}}$. Then there exists an integer $p$ such that $\mathfrak{a}=\mathfrak{h}_p\circ\widetilde{\mathfrak{a}}=\widetilde{\mathfrak{a}}\circ\mathfrak{h}_p$.
\end{proposition}

\begin{proof}
First we remark that for any integer $p$ the automorphisms $\mathfrak{h}_p$ and $\widetilde{\mathfrak{a}}$ commute, i.e., $\mathfrak{h}_p\circ\widetilde{\mathfrak{a}}=\widetilde{\mathfrak{a}}\circ\mathfrak{h}_p$.

\smallskip

Suppose that $\mathfrak{a}(0,0,[0))=(p,p,[k))$ for some integer $p$. Then $\mathfrak{b}=\mathfrak{a}\circ\mathfrak{h}_{-p}$ is an automorphism of the semigroup $\boldsymbol{B}_{\mathbb{Z}}^{\mathscr{F}}$ such that $\mathfrak{b}(0,0,[0))=(0,0,[k))$.  Then the order convexity of the linearly ordered set $L_1=\{(0,0,[0)), (0,0,[1))\}$ implies that the image $\mathfrak{a}(L_1)$ is an order convex chain in $E(\boldsymbol{B}_{\mathbb{Z}}^{\mathscr{F}})$ with the respect to the natural partial order. Remark~\ref{remark-3.1} and the description of the natural partial order on $E(\boldsymbol{B}_{\mathbb{Z}}^{\mathscr{F}})$ (see: Proposition~\ref{proposition-2.4a}) imply that $\mathfrak{b}(0,0,[1))=(1,1,[k-1))$. This completes the proof of the base of induction. Fix an arbitrary $s=2,\ldots,k$ and suppose that $\mathfrak{b}(0,0,[j))=(j,j,[k-j))$ for any $j<s$, which is the assumption of induction. Next, since the linearly ordered set $L_s=\{(0,0,[s-1)), (0,0,[s))\}$ is order convex in $E(\boldsymbol{B}_{\mathbb{Z}}^{\mathscr{F}})$, the image $\mathfrak{a}(L_s)$ is an order convex chain in $E(\boldsymbol{B}_{\mathbb{Z}}^{\mathscr{F}})$, as well. Then the equality $\mathfrak{b}(0,0,[s-1))=(s-1,s-1,[k-s+1))$, Remark~\ref{remark-3.1} and the description of the natural partial order on $E(\boldsymbol{B}_{\mathbb{Z}}^{\mathscr{F}})$ (Proposition~\ref{proposition-2.4a}) imply that $\mathfrak{b}(0,0,[s))=(s,s,[k-s))$ for all $s=2,\ldots,k$.

\smallskip

Fix an arbitrary $s\in\{0,1,\ldots,k\}$.
Since $(1,1,[s))$ is the biggest element of the set of idempotents of $\boldsymbol{B}_{\mathbb{Z}}^{\{[s)\}}$ which are less then $(0,0,[s))$, Remark~\ref{remark-3.1} and the description of the natural partial order on $E(\boldsymbol{B}_{\mathbb{Z}}^{\mathscr{F}})$ (see: Proposition~\ref{proposition-2.4a}) imply that $\mathfrak{b}(1,1,[s))=(1+s,1+s,[k-s))$. Then by induction and presented above arguments we get that $\mathfrak{b}(i,i,[s))=(i+s,i+s,[k-s))$ for any positive integer $i$. Also, since $(-1,-1,[s))$ is the smallest element of the set of idempotents of $\boldsymbol{B}_{\mathbb{Z}}^{\{[s)\}}$ which are greater then $(0,0,[s))$, Remark~\ref{remark-3.1} and the description of the natural partial order on $E(\boldsymbol{B}_{\mathbb{Z}}^{\mathscr{F}})$  imply that $\mathfrak{b}(-1,-1,[s))=(-1+s,-1+s,[k-s))$. Similar, by induction and presented above arguments we get that $\mathfrak{b}(-i,-i,[s))=(-i+s,-i+s,[k-s))$ for any positive integer $i$. This implies that $\mathfrak{b}(i,i,[s))=(i+s,i+s,[k-s))$ for any integer $i$.

\smallskip

Fix any $i,j\in\mathbb{Z}$ and an arbitrary $s=0,1,\ldots,k$. Remark~\ref{remark-3.1} implies that $\mathfrak{b}(i,j,[s))=(m,n,[k-s))$ for some $m,n\in\mathbb{Z}$. By Proposition~21(1) of \cite[Section~1.4]{Lawson=1998} and Lemma~1(4) of \cite{Gutik-Pozdniakova=2021}  we get that $\mathfrak{b}(j,i,[s))=(n,m,[k-s))$. This implies that
\begin{align*}
  \mathfrak{b}(i,i,[s))&=\mathfrak{b}((i,j,[s))\cdot(j,i,[s)))= \\
   &=\mathfrak{b}(i,j,[s))\cdot \mathfrak{b}(j,i,[s))= \\
   &=(m,n,[k-s))\cdot(n,m,[k-s))=\\
   &=(m,m,[k-s))
\end{align*}
and
\begin{align*}
  \mathfrak{b}(j,j,[s))&=\mathfrak{b}((j,i,[s))\cdot(i,j,[s)))= \\
   &=\mathfrak{b}(j,i,[s))\cdot \mathfrak{b}(i,j,[s))= \\
   &=(n,m,[k-s))\cdot(m,n,[k-s))=\\
   &=(n,n,[k-s)),
\end{align*}
and hence we have that $m=i+s$ and $n=j+s$.

\smallskip

Therefore we obtain $\mathfrak{b}(i,j,[s))=(i+s,j+s,[k-s))$ for any $i,j\in\mathbb{Z}$ and an arbitrary $s=0,1,\ldots,k$, which implies that $\mathfrak{b}=\widetilde{\mathfrak{a}}$. Then
\begin{equation*}
\mathfrak{a}=\mathfrak{a}\circ\mathfrak{h}_{-p}\circ\mathfrak{h}_{p}=\mathfrak{b}\circ\mathfrak{h}_{p}=\widetilde{\mathfrak{a}}\circ\mathfrak{h}_{p},
\end{equation*}
which completes the proof of the proposition.
\end{proof}

The following lemma describes the relation between automorphisms $\widetilde{\mathfrak{a}}$ and $\mathfrak{h}_{1}$ of the semigroup $\boldsymbol{B}_{\mathbb{Z}}^{\mathscr{F}}$ in the case when $\mathscr{F}=\{[0),\ldots,[k)\}$.

\begin{lemma}\label{lemma-3.12}
Let $k$ be any positive integer and $\mathscr{F}=\{[0),\ldots,[k)\}$. Then
\begin{equation*}
  \widetilde{\mathfrak{a}}\circ\widetilde{\mathfrak{a}}=\underbrace{\mathfrak{h}_{1}\circ\cdots\circ\mathfrak{h}_{1}}_{k-\hbox{\footnotesize{times}}}= \mathfrak{h}_{k}
\qquad \hbox{and} \qquad \widetilde{\mathfrak{a}}^{-1}=\underbrace{\mathfrak{h}_{1}^{-1}\circ\cdots\circ\mathfrak{h}_{1}^{-1}}_{k-\hbox{\footnotesize{times}}}\circ \; \widetilde{\mathfrak{a}}=\mathfrak{h}_{-k}\circ \widetilde{\mathfrak{a}}.
\end{equation*}
\end{lemma}

\begin{proof}
For any $i,j\in\mathbb{Z}$ and an arbitrary $s=0,1,\ldots,k$ we have that
\begin{align*}
  (\widetilde{\mathfrak{a}}\circ\widetilde{\mathfrak{a}})(i,j,[s))&=\widetilde{\mathfrak{a}}(i+s,j+s,[k-s))= \\
   &=\widetilde{\mathfrak{a}}(i+s+k-s,j+s+k-s,[k-(k-s)))= \\
   &=(i+k,j+k,[s))= \\
   &=\mathfrak{h}_{k}(i,j,[s)),
\end{align*}
Also, by the equality $\widetilde{\mathfrak{a}}\circ\widetilde{\mathfrak{a}}=\mathfrak{h}_{k}$ we get that $\widetilde{\mathfrak{a}}=\mathfrak{h}_{1}^{k}\circ\widetilde{\mathfrak{a}}^{-1}$, and hence
\begin{equation*}
  \widetilde{\mathfrak{a}}^{-1}=\left(\mathfrak{h}_{1}^{k}\right)^{-1}\circ\widetilde{\mathfrak{a}}= \underbrace{\mathfrak{h}_{1}^{-1}\circ\cdots\circ\mathfrak{h}_{1}^{-1}}_{k-\hbox{\footnotesize{times}}}\circ \; \widetilde{\mathfrak{a}}=\mathfrak{h}_{-k}\circ \widetilde{\mathfrak{a}},
\end{equation*}
which completes the proof.
\end{proof}

For any positive integer $k$ we denote the following group $G_k=\left\langle x,y\mid xy=yx, \; y^2=x^k\right\rangle$.

\begin{lemma}\label{lemma-3.13}
For any positive integer $k$ the group $G_k=\left\langle x,y\mid xy=yx, \; y^2=x^k\right\rangle$ is isomorphic to the additive groups of integers $\mathbb{Z}(+)$.
\end{lemma}

\begin{proof}
In the case when $k=2p$ for some positive integer $p$ we have that $y^2=x^{2p}$, and hence $x$ is a generator of $G_k$ such that $y=x^p$.

In the case when $k=2p+1$ for some  $p\in\omega$ we have that $z=y\cdot x^{-k}$ is a generator of $G_k$ such that $x=z^2$ and $y=z^{2p+1}$.
\end{proof}

\begin{theorem}\label{theorem-3.14}
Let $k$ be any positive integer and $\mathscr{F}=\{[0),\ldots,[k)\}$. Then the group $\mathbf{Aut}(\boldsymbol{B}_{\mathbb{Z}}^{\mathscr{F}})$ of automorphisms of the semigroup $\boldsymbol{B}_{Z\mathbb{}}^{\mathscr{F}}$ isomorphic to the group $G_k$, and hence to the additive groups of integers $\mathbb{Z}(+)$.
\end{theorem}

\begin{proof}
By Proposition~\ref{proposition-3.10} for any automorphism $\mathfrak{a}$ of $\boldsymbol{B}_{Z\mathbb{}}^{\mathscr{F}}$ we have that either $\mathfrak{a}(0,0,[0))\in \boldsymbol{B}_{\mathbb{Z}}^{\{[0)\}}$ or $\mathfrak{a}(0,0,[0))\in \boldsymbol{B}_{\mathbb{Z}}^{\{[k)\}}$.

Suppose that $\mathfrak{a}(0,0,[0))\in \boldsymbol{B}_{\mathbb{Z}}^{\{[0)\}}$. Then $\mathfrak{a}(0,0,[0))$ is an idempotent and hence by Lemma~1(2) of \cite{Gutik-Pozdniakova=2021}, $\mathfrak{a}(0,0,[0))=(-p,-p,[0))$ for some integer $p$. Similar arguments as in the proof of Theorem~\ref{theorem-3.6} imply that $\mathfrak{a}=\mathfrak{h}_p=\underbrace{\mathfrak{h}_{1}\circ\cdots\circ\mathfrak{h}_{1}}_{p-\hbox{\footnotesize{times}}}$.

Suppose that $\mathfrak{a}(0,0,[0))\in \boldsymbol{B}_{\mathbb{Z}}^{\{[k)\}}$. Then by Proposition~\ref{proposition-3.11} there exists an integer $p$ such that $\mathfrak{a}=\mathfrak{h}_p\circ\widetilde{\mathfrak{a}}=\widetilde{\mathfrak{a}}\circ\mathfrak{h}_p$.

Since $\widetilde{\mathfrak{a}}$ and $\mathfrak{h}_p$ commute, the above arguments imply that any automorphism $\mathfrak{a}$ of $\boldsymbol{B}_{Z\mathbb{}}^{\mathscr{F}}$ is a one of the following forms:
\begin{itemize}
  \item $\mathfrak{a}=\mathfrak{h}_{p}=(\mathfrak{h}_{1})^p$ for some integer $p$; \qquad or
  \item $\mathfrak{a}=\mathfrak{h}_p\circ\widetilde{\mathfrak{a}}=\widetilde{\mathfrak{a}}\circ\mathfrak{h}_p=\widetilde{\mathfrak{a}}\circ(\mathfrak{h}_{1})^p$ for some integer $p$.
\end{itemize}
This implies that the map $\mathfrak{A}\colon \mathbf{Aut}(\boldsymbol{B}_{\mathbb{Z}}^{\mathscr{F}})\to G_k$ defined by the formulae $\mathfrak{A}((\mathfrak{h}_{1})^p)=x^p$ and  $\mathfrak{A}(\widetilde{\mathfrak{a}}\circ(\mathfrak{h}_{1})^p)=yx^p$, $p\in\mathbb{Z}$, is a group isomorphism. Next we apply Lemma~\ref{lemma-3.12}.
\end{proof}

\section*{Acknowledgements}

The authors acknowledge Alex Ravsky, Taras Banakh and Oleksandr Ganyushkin for their comments and suggestions.


\begin{thebibliography}{00}





\bibitem{Clifford-Preston-1961}
A. H.~Clifford and  G. B.~Preston,
\emph{The algebraic theory of semigroups},
Vol. I., Amer. Math. Soc. Surveys 7, Pro\-vi\-den\-ce, R.I., 1961.

\bibitem{Clifford-Preston-1967}
A. H.~Clifford and  G. B.~Preston,
\emph{The algebraic theory of semigroups},
Vol. II., Amer. Math. Soc. Surveys 7, Provi\-den\-ce, R.I., 1967.

\bibitem{Fihel-Gutik=2011}
I. R. Fihel and O. V. Gutik,
\emph{On the closure of the extended bicyclic semigroup},
Carpathian Math. Publ.  \textbf{3} (2011), no.~2, 131--157.

\bibitem{Green=1951}
J. A. Green,
\emph{On the structure of semigroups},
Ann. Math. (2) \textbf{54} (1951), no.~1, 163--172.


\bibitem{Gutik-Lysetska=2021}
O. Gutik and O. Lysetska,
\emph{On the semigroup $\boldsymbol{B}_{\omega}^{\mathscr{F}}$ which is generated by the family $\mathscr{F}$ of atomic subsets of $\omega$},
Visn. L'viv. Univ., Ser. Mekh.-Mat. \textbf{92} (2021) 34--50 (arXiv:2108.11354).
%

\bibitem{Gutik-Maksymyk-2017}
O. Gutik and K. Maksymyk,
 \emph{On variants of the bicyclic extended semigroup},
Visnyk Lviv. Univ. Ser. Mech.-Mat. \textbf{84} (2017), 22--37.

\bibitem{Gutik-Mykhalenych=2020}
O. Gutik and M. Mykhalenych,
\emph{On some generalization of the bicyclic monoid},
Visnyk Lviv. Univ. Ser. Mech.-Mat. \textbf{90} (2020), 5--19 (in Ukrainian).


\bibitem{Gutik-Mykhalenych=2021}
O. Gutik and M. Mykhalenych,
\emph{On group congruences on the semigroup $\boldsymbol{B}_{\omega}^{\mathscr{F}}$ and its homo\-mor\-phic retracts in the case when the family $\mathscr{F}$ consists of inductive non-empty subsets of~$\omega$},
Visnyk Lviv. Univ. Ser. Mech.-Mat. \textbf{91} (2021), 5--27 (in Ukrainian).


\bibitem{Gutik-Pozdniakova=2021}
O. V. Gutik and I. V. Pozdniakova,
\emph{On the semigroup generating by extended bicyclic semigroup and an $\omega$-closed family},
Mat. Metody Fiz.-Mekh. Polya \textbf{64} (2021), no. 1, 21--34 (in Ukrainian).

\bibitem{Gutik-Prokhorenkova-Sekh=2021}
O. Gutik, O. Prokhorenkova, and D. Sekh,
\emph{On endomorphisms of the bicyclic semigroup and the extended bicyclic semigroup},
Visn. L'viv. Univ., Ser. Mekh.-Mat. \textbf{92} (2021) 5--16 (arXiv:2202.00073) (in Ukrainian).

\bibitem{Harzheim=2005}
E. Harzheim,
\emph{Ordered sets},
Springer, Nre-York, Advances in Math. \textbf{7}, 2005.

\bibitem{Lawson=1998}
M.~Lawson,
\emph{Inverse semigroups. The theory of partial symmetries},
World Scientific, Singapore, 1998.

\bibitem{Lysetska=2020}
O. Lysetska,
\emph{On feebly compact topologies on the semigroup $\boldsymbol{B}_{\omega}^{\mathscr{F}_1}$},
Visnyk Lviv. Univ. Ser. Mech.-Mat. \textbf{90} (2020), 48--56.


\bibitem{Petrich-1984}
M.~Petrich,
\emph{Inverse semigroups},
John Wiley $\&$ Sons, New York, 1984.

\bibitem{Magnus-Karrass-Solitar-1976}
W. Magnus, A. Karrass, and D. Solitar,
\emph{Combinatorial group theory: Presentations of groups in terms of generators and relations},
Dover Publ., 1976.

\bibitem{Wagner-1952}
V.~V. Wagner,
\textit{Generalized groups},
Dokl. Akad. Nauk SSSR \textbf{84} (1952), 1119--1122 (in Russian).

\bibitem{Warne-1968}
R. J. Warne,
\emph{$I$-bisimple semigroups},
Trans. Amer. Math. Soc.  \textbf{130} (1968), no. 3, 367--386.

\end{thebibliography}
\end{document}